\documentclass[12pt]{amsart}
\usepackage{amssymb,amscd}
\usepackage{verbatim}

\usepackage{amsmath,amssymb,graphicx,mathrsfs}   
\usepackage[colorlinks=true,allcolors = blue]{hyperref} 

\usepackage{tikz}
  \usetikzlibrary{calc}
  \usetikzlibrary{matrix}

\usepackage[all]{xy}

\usetikzlibrary{calc} 

\let\frak\mathfrak

\def\>{\relax\ifmmode\mskip.666667\thinmuskip\relax\else\kern.111111em\fi}
\def\<{\relax\ifmmode\mskip-.333333\thinmuskip\relax\else\kern-.0555556em\fi}
\def\vsk#1>{\vskip#1\baselineskip}
\def\vv#1>{\vadjust{\vsk#1>}\ignorespaces}
\def\vvn#1>{\vadjust{\nobreak\vsk#1>\nobreak}\ignorespaces}

  \let\ssize\scriptstyle
\let\sssize\scriptscriptstyle

\let\Medskip\medskip
\def\medskip{\par\Medskip}
\let\Bigskip\bigskip
\def\bigskip{\par\Bigskip}

\let\Maketitle\maketitle
\def\maketitle{\Maketitle\thispagestyle{empty}\let\maketitle\empty}

\newtheorem{thm}{Theorem}[section]

\newtheorem{lem}[thm]{Lemma}
\newtheorem{prop}[thm]{Proposition}

\theoremstyle{definition}                                  
\numberwithin{equation}{section}

\theoremstyle{definition}
\newtheorem*{rem}{Remark}
\newtheorem*{example}{Example}

\let\mc\mathcal
\let\nc\newcommand

\let\al\alpha

\let\ka\kappa
\let\la\lambda

\let\phi\varphi
\let\si\sigma

\let\Om\Omega

\let\der\partial

\let\ox\otimes

\let\geq\geqslant

\let\leq\leqslant

\let\on\operatorname
\let\bi\bibitem
\let\bs\boldsymbol

\def\C{{\mathbb C}}
\def\Z{{\mathbb Z}}
\def\R{{\mathbb R}}

\def\F{{\mathbb F}}   

\def\+#1{^{\{#1\}}}

\def\beq{\begin{equation}}
\def\eeq{\end{equation}}
\def\be{\begin{equation*}}
\def\ee{\end{equation*}}

\nc{\bea}{\begin{eqnarray*}}
\nc{\eea}{\end{eqnarray*}}
\nc{\bean}{\begin{eqnarray}}
\nc{\eean}{\end{eqnarray}}

\let\ga\gamma
\let\Ga\Gamma

\nc{\Il}{{\mc I_{\bs\la}}}
\nc{\bla}{{\bs\la}}
\nc{\Fla}{\F_\bla}
\nc{\tfl}{{T^*\Fla}}
\nc{\GL}{{GL_n(\C)}}
\nc{\GLC}{{GL_n(\C)\times\C^*}}

\let\sd s 

\def\ddk_#1{\kk_{#1}\<\>\frac\der{\der\<\>\kk_{#1}}}

\def\bul{\mathbin{\raise.2ex\hbox{$\sssize\bullet$}}}
\def\intt{\mathchoice
{\mathop{\raise.2ex\rlap{$\,\,\ssize\backslash$}{\intop}}\nolimits}
{\mathop{\raise.3ex\rlap{$\,\sssize\backslash$}{\intop}}\nolimits}
{\mathop{\raise.1ex\rlap{$\sssize\>\backslash$}{\intop}}\nolimits}
{\mathop{\rlap{$\sssize\<\>\backslash$}{\intop}}\nolimits}}

\let\kk q 
\let\cc c

\let\Ko K

\def\GZ/{Gelfand-Zetlin}
\def\KZ/{{\slshape KZ\/}}
\def\qKZ/{{\slshape qKZ\/}}
\def\XXX/{{\slshape XXX\/}}

\nc{\A}{{\mc C}}

\def\Sing{{\on{Sing}}}

\def\slt{{\frak{sl}_2}}

\nc{\hsl}{\widehat{{\frak{sl}_2}}}

\nc{\BC}{{ \mathbb C}}
\nc{\lra}{\longrightarrow}
\nc{\CO}{{\mathcal{O}}}
\nc{\BZ}{{ \mathbb Z}}
\nc{\hfn}{\hat{\frak{n}}}

\def\mm{\text{\bf m}} 
\def\bm{\overline{m}} 

\begin{document}

\hrule width0pt
\vsk->

\title[The $\F_p$-Selberg integral]
{The $\F_p$-Selberg integral}

\author[R.\,Rim\'anyi and  A.\:Varchenko]
{Rich\'ard Rim\'anyi\,$^{\diamond}$ and  Alexander Varchenko\,$^{\star}$}

\maketitle

\begin{center}
{\it $^{\diamond, \star}$ Department of Mathematics, University
of North Carolina at Chapel Hill\\ Chapel Hill, NC 27599-3250, USA\/}

\vsk.5>
{\it $^{ \star}$ Faculty of Mathematics and Mechanics, Lomonosov Moscow State
University\\ Leninskiye Gory 1, 119991 Moscow GSP-1, Russia\/}

\vsk.5>
 {\it $^{ \star}$ Moscow Center of Fundamental and Applied Mathematics
\\ Leninskiye Gory 1, 119991 Moscow GSP-1, Russia\/}

\end{center}

\vsk>
{\leftskip3pc \rightskip\leftskip \parindent0pt \Small
{\it Key words\/}: Selberg integral, $\F_p$-integral, Morris' identity,  Aomoto recursion,
KZ equations,  reduction modulo  $p$

\vsk.6>
{\it 2010 Mathematics Subject Classification\/}: 13A35 (33C60, 32G20) 
\par}

{\let\thefootnote\relax
\footnotetext{\vsk-.8>\noindent
$^\diamond\<${\sl E\>-mail}:\enspace  rimanyi@email.unc.edu, 
supported in part by Simons Foundation grant 523882
\\
$^\star\<${\sl E\>-mail}:\enspace anv@email.unc.edu,
supported in part by NSF grant DMS-1954266}}

\begin{abstract}

We prove an $\F_p$-Selberg integral formula, in which
the $\F_p$-Selberg integral is an element of the finite field $\F_p$ with odd prime number $p$ of  elements.
The formula is motivated by analogy between  multidimensional hypergeometric solutions of the KZ equations and
 polynomial solutions of the same equations reduced modulo $p$.

\end{abstract}

{\small\tableofcontents\par}

\setcounter{footnote}{0}
\renewcommand{\thefootnote}{\arabic{footnote}}

\section{Introduction}

In 1944 Atle Selberg proved the following integral formula:
\bean
\label{clS}
&&
\int_0^1\dots\int_0^1 \prod_{1\leq i<j\leq n} (x_i-x_j)^{2\ga} \prod_{i=1}^nx_i^{\al-1} (1-x_i)^{\beta-1}\ dx_1\dots dx_n
\\
&&
\notag
\phantom{aaaaaa}
=\
\prod_{j=1}^n \frac{\Ga(1+j\ga)}{\Ga(1+\ga)}\,
\frac{\Ga(\al+(j-1)\ga)\,\Ga(\beta+(j-1)\ga)}
{\Ga(\al+\beta + (n+j-2)\ga)}\,,
\eean
see  \cite{Se1, AAR}
\footnote{\phantom{a}In \cite{Se2} Selberg remarks: \ 
``This paper was published with some hesitation, and in Norwegian,
since I was rather doubtful that the results were new. The journal
is one which is read by mathematics-teachers in the gymnasium,
and the proof was written out in some detail so it should be understandable
to someone who knew a little about analytic functions 
and analytic continuation.'' See more in \cite{FW}.}. 
Hundreds of papers are devoted to the generalizations of the Selberg
integral formula and its applications, see for example \cite{AAR, FW} and references 
therein.  There are $q$-analysis versions of the formula,
the generalizations associated with Lie algebras, elliptic versions, finite field versions,
see some references  in \cite{AAR, FW, As, Ha, Ka, Op, Ch, TV1, TV2, TV3, Wa1, Wa2, Sp, R,  FSV, An, Ev}.
In the finite field versions,
 one considers additive and multiplicative characters of a finite field, which 
map the  field to the field of
complex numbers, and forms an analog of equation \eqref{clS}, in which  both sides are complex numbers.
The simplest of such formulas is the classical relation between Jacobi and Gauss 
sums, see \cite{AAR, An, Ev}.

\vsk.2>
In this paper we suggest another  version of the Selberg integral formula, in which the
$\F_p$-Selberg integral is an element of the  finite field $\F_p$ with an odd prime number $p$  of elements,
see Theorem \ref{thm nd}.

\vsk.2>
Our motivation comes from the theory of the Knizhnik-Zamolodchikov 
(KZ) equations, see \cite{KZ, EFK}. These are the
  systems of linear differential equations, satisfied by
conformal blocks on the sphere in the WZW model of conformal field theory.  
The KZ equations were solved in multidimensional hypergeometric integrals in \cite{SV1}, see also \cite{V1,V2}.
The following general principle was formulated in \cite{MV}: \ if an example of the  KZ 
type equations has a one-dimensional space of solutions, then the 
corresponding multidimensional hypergeometric integral can be evaluated explicitly. 
As an illustration of that principle in \cite{MV},  an example of  KZ equations with a
one-dimensional space of solutions 
was considered, the corresponding multidimensional hypergeometric integral was reduced to the 
Selberg integral and  then evaluated by formula \eqref{clS}. Other illustrations see in \cite{FV, FSV, TV1, TV2, TV3, V3, RTVZ}.

\vsk.2>

Recently in \cite{SV2} the KZ equations were considered modulo a prime number $p$ and polynomial solutions of the 
reduced equations were constructed, see also \cite{SlV, V4, V5,V6,V7}. The construction
 is analogous to the construction
of the multidimensional hypergeometric solutions,
and the constructed polynomial solutions were called 
the  $\F_p$-hypergeometric solutions.

\vsk.2>
In this paper we consider the reduction modulo $p$ of the same example of the KZ equations, 
that led in \cite{MV} to the  Selberg integral. The space of solutions of the reduced KZ equations is still
one-dimensional and, according to the principle,
we  may expect that the corresponding $\F_p$-hypergeometric solution is related to 
a Selberg type formula. Indeed 
we have evaluated that $\F_p$-hypergeometric solution by analogy with the evaluation of the Selberg integral and
obtained our $\F_p$-Selberg integral formula in  Theorem \ref{thm nd}.

\vsk.2>  
The paper contains three proofs of our $\F_p$-Selberg integral formula. There might be more proofs. 
It would be interesting to see if our formula can be deduced from the known relations between the multidimensional
Gauss and Jacobi sums, see for example \cite[Section 8.11]{AAR}.

\smallskip

The paper is organized as follows. In Section \ref{sec 2} we collect useful facts.
In Section \ref{sec 3} we introduce the notion of $\F_p$-integral and discuss the
integral formula for the $\F_p$-beta integral.  In Section \ref{sec 4} we formulate our main result, Theorem
\ref{thm nd}, and prove it by developing an $\F_p$-analog of Aomoto's recursion,
 defined in \cite{Ao} for the Selberg integral.
In Section \ref{sec 5} we give another proof of Theorem \ref{thm nd}, based on Morris' identity, which 
 is deduced from the classical  Selberg integral formula \eqref{clS} in \cite{Mo}.
 In Section \ref{sec 6} we  sketch a third proof of Theorem \ref{thm nd} based on a 
  combinatorial identity, also deduced from the Selberg integral formula \eqref{clS}.
In Section \ref{sec 7} we discuss in more detail how our $\F_p$-Selberg integral 
formula is related to the $\F_p$-hypergeometric solutions
of KZ equations reduced modulo $p$.

\smallskip
The authors thank C.\,Bir\'o, I.\,Cherednik, P.\,Etingof, E.\,Rains, A.\,Slinkin for useful discussions.

\section{Preliminary remarks}
\label{sec 2}

\subsection{Lucas' Theorem}

\begin{thm} [\cite{L}] \label{thm L}
For nonnegative integers $m$ and $n$ and a prime $p$,
 the following congruence relation holds:
\bean
\label{Lucas}
\binom{n}{m}\equiv \prod _{i=0}^a \binom{n_i}{m_i}\quad (\on{mod}\ p),
\eean
where $m=m_{b}p^{b}+m_{b-1}p^{b-1}+\cdots +m_{1}p+m_{0}$ 
and  $n=n_{b}p^{b}+n_{b-1}p^{b-1}+\cdots +n_{1}p+n_{0}$
are the base $p$ expansions of $m$ and $n$ respectively. This uses the convention that 
$\binom{n}{m}=0$ if $n<m$.
\qed
\end{thm}

\subsection{Binomial lemma}

\begin{lem} [\cite{V7}]
\label{lem ga}
Let $a,b$ be positive integers such that $a<p$, $b<p$, $p\leq a+b$. Then we have an identity in $\F_p$,
\bean
\label{p-ga}
  b\,\binom{b-1}{a+b-p} 
  &=&
b\,\binom{b-1}{p-a-1}
=
 (-1)^{a+1}\,\frac{a! \,b!}
{(a+b-p)!}\,.
\eean
\qed

\end{lem}

\begin{proof}
We have
\bea
&&
b\binom{b-1}{p-a-1} = \frac{b(b-1) \cdots(a+b-p+1)}
{1\cdots (p-a-1)}
\\
&&
\phantom{aaa}
= \frac{b \cdots (a+b-p+1) (a+b-p)! a!}
{(-1)^{p-a-1} (p-1)(p-2) \cdots  (a+1)\,a!\,(a+b-p)!}
= (-1)^{a+1}\,\frac{a! \,b!}
{(a+b-p)!} \,.
\eea
\end{proof}

\subsection{Cancellation of factorials}

\begin{lem}
\label{lem ca}

If $a,b$ are nonnegative integers and  $a+b=p-1$, then in $\F_p$ we have
\bean
\label{ca}
a!\,b!\,=\,(-1)^{a+1}\,.
\eean

\end{lem}

\begin{proof} We have $a!=(-1)^a(p-1)\dots(p-a)$ and $p-a= b+1$. Hence
$a!\,b! = (-1)^a(p-1)!= (-1)^{a+1}$ by Wilson's Theorem.
\end{proof}

\section{$\F_p$-Integrals}
\label{sec 3}

\subsection{Definition}
Let $p$ be an odd prime number and $M$  an $\F_p$-module. Let $P(x_1,\dots,x_k)$ be a polynomial
with coefficients in $M$,
\bean
\label{St}
P(x_1,\dots,x_k) = \sum_{d}\, c_d \,x_1^{d_1}\dots x_k^{d_k}.
\eean
Let $l=(l_1,\dots,l_k)\in \Z_{>0}^k$. The coefficient
$c_{l_1p-1,\dots,l_kp-1}$ is called the {\it $\F_p$-integral over the cycle $[l_1,\dots,l_k]_p$}
and is denoted by $\int_{[l_1,\dots,l_k]_p} P(x_1,\dots,x_k)\,dx_1\dots dx_k$.

\begin{lem}
\label{lem sym}

For $i=1,\dots,k-1$ we have
\bean
\label{sym}
&&
\int_{[l_1,\dots,l_{i+1}, l_i,\dots, l_k]_p} P(x_1,\dots,x_{i+1},x_i,\dots,x_k)dx_1\dots dx_k
\\
\notag
&&
\phantom{aaaaa}
=\ 
\int_{[l_1,\dots,l_k]_p} P(x_1,\dots,x_k)\,dx_1\dots dx_k\,.
\eean
\qed
\end{lem}

\begin{lem}
\label{lem St}
For any $i=1,\dots,k$, we have
\bea
\int_{[l_1,\dots,l_k]_p} \frac{\der P}{\der x_i}(x_1,\dots,x_k) = 0\,.
\eea
\qed
\end{lem}

\subsection{$\F_p$-Beta integral}
For nonnegative integers the classical beta integral formula says
\bean
\label{bk}
\int_{0}^{1} x^{a}(1-x)^{b} dx =\frac{a!\,b!}{(a+b+1)!}\,.
\eean

\begin{thm} [\cite{V7}]
\label{thm bfun}
Let
$a<p$, $b<p$,  $p-1\leq a+b$. 
Then in $\F_p$ we have
\bean
\label{bf1}
\int_{[1]_p} x^a(1-x)^b dx\,
=  \, - \,\frac{a!\,b!}{(a+b-p+1)!}\,.
\eean
If $a+b<p-1$, then 
\bean
\label{bf2}
\int_{[1]_p} x^a(1-x)^b dx\,
=  0\,.
\eean

\end{thm}

\begin{proof}
We have
$
x^a(1-x)^b = \sum_{k=0}^b(-1)^k \binom{b}{k} x^{k}\,,
$
and need  $a+k=p-1$. Hence $k=p-1-a$ and 
\bea
\int_{[1]_p} x^a(1-x)^bdx \,
= (-1)^{p-1-a}\binom{b}{p-1-a}.
\eea
Now Lemma \ref{lem ga} implies \eqref{bf1}. Formula \eqref{bf2} is clear.
\end{proof}

\section{$n$-Dimensional $\F_p$-Selberg integral}
\label{sec 4}

\subsection{$n$-Dimensional integral formulas}

The $n$-dimensional Selberg integral formulas for nonnegative integers $a,b,c$ are
\bean
\label{cSn}
&&
\int_0^1\dots\int_0^1 \prod_{1\leq i<j\leq n} (x_i-x_j)^{2c} \prod_{i=1}^nx_i^a (1-x_i)^b\ dx_1\dots dx_n
\\
&&
\notag
\phantom{aaaaaa}
=
\prod_{j=1}^n \frac{(jc)!}{c!}\,
\frac{(a+(j-1)c)!\,(b+(j-1)c)!}
{(a+b + (n+j-2)c+1)!}\,,
\eean
and for $k=1,\dots,n-1$,
\bean
\label{cAn}
&&
\int_0^1\dots\int_0^1\prod_{i=1}^kx_i
\prod_{1\leq i<j\leq n} (x_i-x_j)^{2c} \prod_{i=1}^nx_i^a (1-x_i)^b\ dx_1\dots dx_n
\\
&&
\notag
\phantom{aa}
=\,
\prod_{j=1}^k \frac{a+(n-j)c+1}{a+b+(2n-j-1)c+1}\,
\prod_{j=1}^n \frac{(jc)!}{c!}\,
\frac{(a+(j-1)c)!\,(b+(j-1)c)!}
{(a+b + (n+j-2)c+2)!}\,,
\eean
\cite{Se1, Ao, AAR}.

\begin{thm}
\label{thm nd}

Assume that $a,b,c$ are nonnegative integers such that
\bean
\label{abc n eq}
 p-1\leq a+b+(n-1)c,
\qquad a+b+(2n-2)c < 2p-1\ .
\eean
Then we have an integral formula in $\F_p$:
\bean
\label{main n}
&&
\int_{[1,\dots,1]_p} \prod_{1\leq i<j\leq n} (x_i-x_j)^{2c} \prod_{i=1}^nx_i^a (1-x_i)^b\ dx_1\dots dx_n
\\
&&
\notag
\phantom{aaaaaa}
=
(-1)^n\,\prod_{j=1}^n \frac{(jc)!}{c!}\,
\frac{(a+(j-1)c)!\,(b+(j-1)c)!}
{(a+b + (n+j-2)c+1-p)!}\,.
\eean
Also, if  $k=1,\dots,n-1$, and
\bean
\label{abc ine}
 p-1\leq a+b+(n-1)c,
\qquad 
a+b+(2n-2)c < 2p-2\ ,
\eean
then
\bean
\label{mainIn}
&&
\int_{[1,\dots,1]_p}
\prod_{i=1}^kx_i\prod_{1\leq i<j\leq n} (x_i-x_j)^{2c} \prod_{i=1}^nx_i^a (1-x_i)^b\ dx_1\dots dx_n
\\
&&
\notag
\phantom{a}
= \,(-1)^n\,
\prod_{j=1}^k \frac{a+(n-j)c+1}{a+b+(2n-j-1)c+2}\,
\prod_{j=1}^n \frac{(jc)!}{c!}\,
\frac{(a+(j-1)c)!\,(b+(j-1)c)!}
{(a+b + (n+j-2)c+1-p)!}\,.
\eean
\end{thm}

The first proof of Theorem \ref{thm nd} is given  in  
Sections \ref{sec mp} - \ref{sec Aom rec}, the second in Section \ref{sec 5},
 and the third one is sketched in Section \ref{sec 6}.

\begin{rem}
Formula \eqref{main n} can be rewritten as 
\bean
\label{main n sum}
&&
\sum_{x_1,\dots,x_n\in\F_p}\ \prod_{1\leq i<j\leq n} (x_i-x_j)^{2c} \prod_{i=1}^nx_i^a (1-x_i)^b\ dx_1\dots dx_n
\\
&&
\notag
\phantom{aaaaaa}
=
\,\prod_{j=1}^n \frac{(jc)!}{c!}\,
\frac{(a+(j-1)c)!\,(b+(j-1)c)!}
{(a+b + (n+j-2)c+1-p)!}\,.
\eean

\end{rem}

\begin{rem}
The fact that the $\F_p$-Selberg integral in the left-hand side of \eqref{main n} equals an explicit alternating
product  in the right-hand side of \eqref{main n} is surprising. But even more surprising is the fact that
the alternating product in the right-hand side is exactly  the alternating product staying in the classical formula
\eqref{cSn} with just several of factorials shifted by $p$.

\end{rem}

\begin{rem}
The Selberg integral \eqref{cSn} is related to the $\slt$ KZ differential equations, see Section \ref{sec 7},
 and is called the Selberg integrals of type $A_1$. 
The Selberg integrals of types $A_n$, related to the $\frak{sl}_{n+1}$ KZ differential equations, are introduced in
\cite{TV3, Wa1, Wa2}.
\vsk.2>

We call the $\F_p$-integral \eqref{main n}  the $\F_p$-Selberg integral of type $A_1$.
The $\F_p$-Selberg integral of type $A_n$, $n>1$, are introduced in \cite{RV}.
The $\F_p$-Selberg integral formula of type $A_n$ is deduced in \cite{RV} for the $\F_p$-Selberg integral formula
\eqref{main n} by induction on $n$.

\end{rem}

\begin{rem}
The integral analogous to \eqref{main n} but with $x_i-x_j$ factors raised to an odd power vanishes: 
\begin{equation}\label{odd case}
\int_{[1,\dots,1]_p} \prod_{1\leq i<j\leq n} (x_i-x_j)^{2c+1} \prod_{i=1}^nx_i^a (1-x_i)^b\ dx_1\dots dx_n=0.
\end{equation}
Indeed, after expanding the $(x_1-x_2)^{2c+1}$ factor, the integral \eqref{odd case} equals
\[
\sum_{m=0}^{2c+1}(-1)^{m+1} \binom{2c+1}{m} \int_{[1,\dots,1]_p} x_1^{a+m} x_2^{a+(2c+1-m)} f(x_1,\ldots,x_n) \ dx_1\dots dx_n=0,
\]
with $f$ symmetric in $x_1$ and $x_2$. The terms corresponding to $m$ and $2c+1-m$ cancel each other, making the sum 0.
\end{rem}

\subsection{Auxiliary lemmas}
\label{sec mp}
Denote
\bean
\label{P2}
P_n(a,b,c) \,=\,(-1)^n\,
\prod_{j=1}^n \frac{(jc)!}{c!}\,
\frac{(a+(j-1)c)!\,(b+(j-1)c)!}
{(a+b + (n+j-2)c+1-p)!}\,.
\eean
The polynomial
\bea
\Phi(x_1,\dots, x_n,a,b,c)= 
 \prod_{1\leq i<j\leq n} (x_i-x_j)^{2c} \prod_{i=1}^nx_i^a (1-x_i)^b
 \eea
is called the {\it master polynomial}. Denote
\bea
S_n(a,b,c) 
&=&
 \int_{[1,\dots,1]_p} \prod_{1\leq i<j\leq n} (x_i-x_j)^{2c} \prod_{i=1}^nx_i^a (1-x_i)^b dx_1\dots dx_n\,,
\\
S_{k,n}(a,b,c) 
&=&
 \int_{[1,\dots,1]_p} \prod_{i=1}^k x_i
 \prod_{1\leq i<j\leq n} (x_i-x_j)^{2c} \prod_{i=1}^nx_i^a (1-x_i)^b\, dx_1\dots dx_n\,,
\eea
for $k=0,\dots,n$.  Then $S_{0,n}(a,b,c)= S_n(a,b,c)$, $S_{n,n}(a,b,c)= S_n(a+1,b,c)$.  
By \eqref{sym}, we also have
\bea
S_{k,n}(a,b,c) = 
 \int_{[1,\dots,1]_p} \prod_{i=1}^k x_{\si_i}
 \prod_{1\leq i<j\leq n} (x_i-x_j)^{2c} \prod_{i=1}^nx_i^a (1-x_i)^b\, dx_1\dots dx_n\,
\eea
for any $1\leq \si_1<\dots<\si_k\leq n$.

\begin{lem}
\label{lem b=b+p}
We have $S_n(a,b+p,c) = S_n(a,b,c)$.

\end{lem}

\begin{proof}
We have $(1-x_i)^{b+p} = (1-x_i)^{b}(1-x_i)^{p} =(1-x_i)^{b}(1-x_i^p)$. Hence the factors
$(1-x_i)^{b}$ and $(1-x_i)^{b+p}$ contribute to the coefficient of $x_i^{p-1}$ in the same way.
\end{proof}

\begin{lem}
\label{lem n<p}
If $a+b+(2n-2)c < 2p-2$ and $c>0$, then $n<p$.
\qed
\end{lem}

\begin{lem}
\label{lem S=01}
If $a+b+(n-1)c < p-1$, then $S_n(a,b,c)=0$.
\end{lem}

\begin{proof}
The coefficient 
of $x_1^{p-1}\dots x_n^{p-1}$ in the expansion of $\Phi(a,b,c)$ equals zero.
\end{proof}

\begin{lem}
\label{lem S=0}
If $p\leq a+(n-1)c$, then $S_n(a,b,c)=0$.

\end{lem}

\begin{proof}
Expand $\Phi(x,a,b,c)$ in monomials $x_1^{d_1}\dots x_n^{d_n}$.
If $p\leq a+(n-1)c$, then  each monomial $x_1^{d_1}\dots x_n^{d_n}$
in the expansion has at least one of $d_1,\dots,d_n$ greater than $p-1$,
Hence the coefficient of $x_1^{p-1}\dots x_n^{p-1}$ in the expansion equals zero.
\end{proof}

\begin{lem}
\label{lem sym}
If $a+b+(2n-2)c<2p-1$, then $S_n(a,b,c)=S_n(b,a,c)$.
\end{lem}

\begin{proof}
Expand $\Phi(x,a,b,c)$ in monomials $x_1^{d_1}\dots x_n^{d_n}$.
If $a+b+(2n-2)c <2p-1$, then

\vsk.2>
\noindent
(a)\  for each monomial $x_1^{d_1}\dots x_n^{d_n}$
in the expansion all of $d_1,\dots,d_n$ are less than $2p-1$.

\vsk.2>

We also have
\bea
\Phi(1-y_1,\dots,1-y_n,a,b,c)= 
\prod_{1\leq i<j\leq n} (y_i-y_j)^{2c} \prod_{i=1}^ny_i^a (1-y_i)^b\,.
\eea
This transformation does not change the $\F_p$-integral due to Lucas' Theorem
and property (a), see a similar reasoning in  the proof of \cite[Lemma 5.2]{V5}.
\end{proof}

\subsection{Case $a+b+(n-1)c=p-1$}
\label{sec indn}

\begin{lem}
\label{lem ind n}
If $a+b+(n-1)c=p-1$, then
\bean
\label{Sn ind}
S_n(a,b,c) = (-1)^{bn+cn(n-1)/2}\,\frac{(cn)!}{(c!)^n}\, .
\eean

\end{lem}

\begin{proof}
If $a+b+(n-1)c=p-1$, then $S_n(a,b,c)$ equals $(-1)^{bn}$ multiplied by the coefficient of $(x_1\dots x_n)^c$ in 
$\prod_{1\leq i<j\leq n} (x_i-x_j)^{2c}$, which  equals $(-1)^{cn(n-1)/2}\frac{(cn)!}{(c!)^n}$ by Dyson's formula
\bean
\label{DF}
\on{C.T.} \prod_{1\leq i<j\leq n}(1-x_i/x_j)^c (1-x_j/x_i)^c = \frac{(cn)!}{(c!)^n}\,.
\eean
Here C.T. denotes the constant term. See the formula in \cite[Section 8.8]{AAR}.
\end{proof}

\begin{lem}
\label{lem ind n2}
If $a+b+(n-1)c=p-1$, then
\bean
\label{pn ind}
\phantom{aaaa}
P_n(a,b,c) \,=\, (-1)^{bn+cn(n-1)/2}\,\frac{(cn)!}{(c!)^n}\, .
\eean

\end{lem}

\begin{proof}
We have
\bea
P_n(a,b,c) 
&=& (-1)^n\,
\prod_{j=1}^n \frac{(jc)!}{c!}\,
\frac{(a+(j-1)c)!\,(b+(j-1)c)!}
{(a+b + (n+j-2)c+1-p)!}
\\
&=& 
(-1)^n\,\prod_{j=1}^n \frac{(jc)!}{c!}\,
\frac{(a+(j-1)c)!\,(b+(j-1)c)!}
{c!\,(2c)!\dots ((n-1)c)!}\,.
\eea
By Lemma \ref{lem ga} we have  $a!\,(b+(n-1)c)! = (-1)^{b+(n-1)c+1}$,
$(a+c)!\,(b+(n-2)c)! = (-1)^{b+(n-2)c+1}$, and so on. This proves the lemma.
\end{proof}

Lemmas \ref{lem ind n} and \ref{lem ind n2} prove formula \eqref{main n} for 
$a+b+(n-1)c=p-1$.

\subsection{Aomoto recursion} 
\label{sec Aom rec}

We follow the paper \cite{Ao},  where recurrence 
relations were developed for the classical Selberg integral.
See also \cite[Section 8.2]{AAR}.

\vsk.2>

Using Lemma \ref{lem St}, for $k=1,\dots,n$ we have
\bean
\label{r1n}
0
&=&
\int_{[1,\dots,1]_p}\frac{\der}{\der x_1}\Big[(1-x_1) \prod_{i=1}^kx_i\,\Phi(x,a,b,c)\Big] dx_1\dots dx_n
\\
\notag
&=&
(a+1)\int_{[1,\dots,1]_p}(1-x_1) \prod_{i=2}^kx_i\,\Phi(x,a,b,c)\, dx_1\dots dx_n
\\
\notag
&&
- (b+1)
\int_{[1,\dots,1]_p} \prod_{i=1}^kx_i\,\Phi(x,a,b,c)\, dx_1\dots dx_n
\\
\notag
&&
+2c \int_{[1,\dots,1]_p}\sum_{j=2}^n\frac{1-x_1}{x_1-x_j} \prod_{i=1}^kx_i\,\Phi(x,a,b,c)\, dx_1\dots dx_n\,.
\eean

\begin{lem}
\label{lem Ik}
The $\F_p$-integral
\bean
\label{Ik1n}
\int_{[1,\dots,1]_p}\frac{1}{x_1-x_j}\,\prod_{i=1}^kx_i\,\Phi(x,a,b,c)\, dx_1\dots dx_n
\eean
equals $0$ if $2\leq j\leq k$ and equals $S_{k-1,n}/2$ if $k<j\leq n$.  
The $\F_p$-integral
\bean
\label{Ik2n}
&&
\int_{[1,\dots,1]_p}\frac{x_1}{x_1-x_j}\,\prod_{i=1}^kx_i\,\Phi(x,a,b,c)\, dx_1\dots dx_n
\eean
equals $S_{k,n}/2$ if $2\leq j\leq k$ and equals $S_{k,n}$ if $k<j\leq n$.  
\end{lem}

\begin{proof}
By Lemma \ref{lem sym} each of these integrals does not change if $x_1,x_j$ are permuted.
The four statements of the lemma hold since $\frac{x_1x_j}{x_1-x_j}+\frac{x_1x_j}{x_j-x_1}=0$,
$\frac{x_1}{x_1-x_j}+\frac{x_j}{x_j-x_1}=1$,
$\frac{x_1^2x_j}{x_1-x_j}+\frac{x_1x_j^2}{x_j-x_1}=x_1x_j$,
$\frac{x_1^2}{x_1-x_j}+\frac{x_j^2}{x_j-x_1}=x_1+x_j$, respectively.
\end{proof}

\begin{lem}
\label{lem Reln}
For  $k=1,\dots,n$ we have
\bean
\label{re1n}
S_{k,n}
\,=\,
 \frac{a+(n-k)c+1}{a+b+(2n-k-1)c+2}\,S_{k-1,n}\,.
 \eean
\end{lem}

\begin{proof}
Using Lemma \ref{lem Ik} we rewrite \eqref{r1n} as
\bea
0 = (a+1)S_{k-1,n} - (a+b+2)S_{k,n} + c(n-k)S_{k-1,n} -
c(2n-k-1)S_{k,n}\,.
\eea
\end{proof}

\subsection{Proof of Theorem \ref{thm nd}} 
\label{sec ptnd}

Theorem \ref{thm nd} is proved by induction on $a$ and $b$. The base induction step
$a+b+(n-1)c=p-1$ is proved in Section \ref{sec indn}.

Lemma \ref{lem Reln} gives
\bea
S_n(a+1,b,c)\, =\, \,S_n(a,b,c)\, 
\prod_{j=1}^n \frac{a+(n-j)c+1}{a+b+(2n-j-1)c+1}\,
\,.
\eea
Together with the symmetry $S_n(a,b,c)=S_n(b,a,c)$ this gives formula
\eqref{main n}. Then formula \eqref{re1n} gives formula \eqref{mainIn}.
Theorem \ref{thm nd} is proved.

\subsection{Relation to Jacobi polynomials}
The statements \eqref{mainIn} for different values of $k$ can be organized to just one equality which involves a Jacobi polynomial
--
like it was done by K.\,Aomoto in \cite{Ao} for the classical Selberg integral. Recall that the degree $n$ Jacobi polynomial is
\[
P^{(n)}_{\alpha,\beta}(x)=\frac{1}{n!}\sum_{\nu=0}^n \binom{n}{\nu} \prod_{i=1}^\nu(n+\alpha+\beta+i)\prod_{i=\nu+1}^n(\alpha+i) \left(\frac{x-1}{2}\right)^\nu.
\]

\begin{prop} Assuming inequalities \eqref{abc ine} let $\alpha=(a+1)/c-1$, $\beta=(b+1)/c-1$. Then
\begin{equation}\label{jacobi}
\int_{[1,\ldots,1]_p} \prod_{i=1}^n (x_i-t) \cdot \Phi(x,a,b,c) dx_1\ldots dx_n
=
\frac{n!c^n \cdot S_n(a,b,c)}{\prod_{i=n-1}^{2n-2} (a+b+ic+2)}
\cdot
P^{(\alpha,\beta)}_n(1-2t).
\end{equation}
\end{prop}

The proof is the same is in \cite{Ao}: After expanding $\prod_{i=1}^n (x_i-t)$ we have the sum of integrals of the type
\[
\int_{[1,\ldots,1]_p} x_{\sigma_1}x_{\sigma_2}\ldots x_{\sigma_k} \Phi(x,a,b,c) dx_1\ldots dx_n,
\]
which --- by symmetry \eqref{sym} --- are equal to $S_{k,n}(a,b,c)$. Substituting 
\[
S_{k,n}(a,b,c)= S_n(a,b,c)\cdot \prod_{j=1}^k\frac{a+(n-j)c+1}{a+b+(2n-j-1)c+2}\]
from \eqref{main n} and  \eqref{mainIn} yields \eqref{jacobi}.

\section{$\F_p$-Selberg integral from Morris' identity}
\label{sec 5}

\subsection{Morris' identity}

In this section we work out the integral formula \eqref{main n} for the $\F_p$-Selberg integral from
Morris' identity. 
Suppose that $\al,\beta,\ga$ are nonnegative integers. Then
\bean
\label{Mid}
&&
\on{C.T.}\,\prod_{i=1}^n (1-x_i)^\al (1-1/x_i)^\beta \prod_{1\leq j\ne k\leq n} (1-x_j/x_k)^\ga
\\
\notag
&&
\phantom{a}
=\, \ \prod_{j=1}^n \frac{(j\ga)!}{\ga!}\,
\frac{(\al+\beta +(j-1)\ga)!}
{(\al+(j-1)\ga)!\,(\beta +(j-1)\ga)!}\,.
\eean
 Morris identity was deduced in \cite{Mo} from the integral formula \eqref{cSn}
for the classical Selberg integral, see \cite[Section 8.8]{AAR}.

The left-hand side of  \eqref{Mid} can be written as 
\bean
\label{Mim}
\on{C.T.}\,(-1)^{\binom{n}{2}\ga + n\beta}
\prod_{1\leq i<j\leq n}(x_i-x_j)^{2\ga}
\prod_{i=1}^n x_i^{-\beta-(n-1)\ga}(1-x_i)^{\al+\beta}\,,
\eean
while
\bean
\label{SnM}
S_n(a,b,c) \,=\,
\on{C.T.}\,
\prod_{1\leq i<j\leq n}(x_i-x_j)^{2c}
\prod_{i=1}^n x_i^{a+1-p}(1-x_i)^{b}\,,
\eean
where the constant term is projected to $\F_p$.

Putting $a +1-p = -\beta-(n-1)\ga, \,  b=\al+\beta\,,  c= \ga$,
or
\bean
\label{al a}
\al= a+b+(n-1)c+1-p,\qquad \beta = p-a-(n-1)c-1,\qquad \ga= c. 
\eean
we obtain the following theorem.

\begin{thm}
\label{thm SM} 
If the nonnegative integers $a,b,c$ satisfy the inequalities
\bean
\label{Mineq}
p-1\leq a+b+(n-1)c,\qquad a+(n-1)c\leq p-1,
\eean
 then the $\F_p$-Selberg integral is given by the formula:
\bean
\label{SnMM}
&&
S_n(a,b,c) \,=\, (-1)^{\binom{n}{2}c + na}
 \\
 \notag
&&
\phantom{aaa}
\times\,
 \prod_{j=1}^n \frac{(jc)!}{c!}\,
\frac{(b +(j-1)c)!}
{(p-a-(n-j)c-1)!\,(a+b  +(n+j-2)c+1-p)!}\,,
\eean
where the integer on the right-hand side is projected
to $\F_p$.
\end{thm}

\begin{lem}
\label{lem S=M} If both inequalities \eqref{abc n eq} and \eqref{Mineq} hold,
that is, if 
\bean
\label{abc ines}
&
 p-1\leq a+b+(n-1)c,
\qquad a+b+(2n-2)c< 2p-1\ ,
\\
\label{abs ine2}
&a+(n-1)c\leq p-1,
\eean
 then 
in $\F_p$ we have
\bean
\label{m=s}
&&
(-1)^{\binom{n}{2}c + na}
 \prod_{j=1}^n \frac{(jc)!}{c!}\,
\frac{(b +(j-1)c)!}
{(p-a-(n-j)c-1)!\,(a+b  +(n+j-2)c)!}\,,
\\
\notag
&&
\phantom{aa}
=\,(-1)^n\,\prod_{j=1}^n \frac{(jc)!}{c!}\,
\frac{(a+(j-1)c)!\,(b+(j-1)c)!}
{(a+b + (n+j-2)c+1-p)!}\,,
\eean
and hence  \eqref{SnMM} 
\bean
\label{main nnn}
S_n(a,b,c)=
(-1)^n\,\prod_{j=1}^n \frac{(jc)!}{c!}\,
\frac{(a+(j-1)c)!\,(b+(j-1)c)!}
{(a+b + (n+j-2)c+1-p)!}\,.
\eean

\end{lem}

Notice that by Lemma \ref{lem S=0} we have $S_n(a,b,c)=0$ if inequality \eqref{abs ine2}
does not hold.

\begin{proof} We have
\bea
\prod_{j=1}^n 
\frac{1}
{(p-a-(n-j)c-1)!}
&=&
\prod_{j=1}^n 
\frac{(a+(n-j)c)!}
{(p-a-(n-j)c-1)!\,(a+(n-j)c)!}
\\
&=&
\prod_{j=1}^n 
(-1)^{a+(n-j)c+1}(a+(n-j)c)!\,,
\eea
by Lemma \ref{lem ca}. This implies the Lemma \ref{lem S=M}.
\end{proof}

\subsection{More on values of $S_n(a,b,c)$}

\begin{thm}
\label{thm SMnb}
If inequalities \eqref{Mineq} hold and $a=p-1-(n-1)c-k$, then
\bean
\label{SMnb}
&&
S_n(p-1-(n-1)c-k,b,c) \,=\, (-1)^{\binom{n}{2}c + na} \,\frac{(nc)!}{(c!)^n}\,
 \prod_{j=1}^n 
\frac{\binom{b+(j-1)c}{k}}{\binom{(j-1)c+k}{k}}\,,
\eean
where the integer in the right-hand side is projected
to $\F_p$.
\qed
\end{thm}

Notice that the  projections to $\F_p$ of the binomial coefficients $\binom{b+(j-1)c}{k}$
 can be calculated by Lucas's Theorem and both integers in the binomial coefficients
$\binom{(j-1)c+k}{k}$ are nonnegative and less than $p$.
\begin{proof}

We have
\bea
&&
\frac{(\al+\beta +(j-1)\ga)!}
{(\al+(j-1)\ga)!\,(\beta +(j-1)\ga)!}
=\binom{\al+\beta +(j-1)\ga}
{\beta}
\prod_{i=1}^{(j-1)\ga} \frac1{\beta + i}\,
\\
&&
\phantom{aaa}
=\,
\binom{b +(j-1)c}
{p-a-(n-1)c-1}
\prod_{i=1}^{(j-1)c} \frac1{p-a-(n-1)c-1 + i}\,.
\eea
If $a=p-1-(n-1)c-k$, then this equals
\bea
\binom{b +(j-1)c}
{k}
\prod_{i=1}^{(j-1)c} \frac1{k + i}
&=&
\binom{b +(j-1)c}
{k}
 \frac{k!}{((j-1)c)!\,\prod_{i=1}^k ((j-1)c+i)}\,
\\
&
=&
\frac 1{((j-1)c)!}\frac{\binom{b +(j-1)c}
{k}}{\binom{(j-1)c+k}{k}}\,.
\eea
Substituting this to \eqref{SnMM} we obtain \eqref{SMnb}.
\end{proof}

\begin{example}
Formula \eqref{SMnb} gives
\bea
S_2(p-c-1,b,c) = (-1)^c \binom{2c}{c}\,,
\qquad
S_2(p-c-2,b,c) = (-1)^c \binom{2c}{c}\,\frac{b(b+c)}{c+1}\,,
\eea
and so on. Notice that these values are not given by Theorem \ref{thm nd}. See more examples in Figure 1.
\end{example}

\usetikzlibrary{calc}

\begin{figure}
\begin{tikzpicture}[xscale=.45, yscale=.45]

\coordinate (a) at (0,0);

\path [fill=yellow] ($(a)+(10,12.5)$) -- ($(a)+(-.5,2)$)  -- ($(a)+(-.5,-.5)$)  -- 
($(a)+(.5,-.3)$)  -- ($(a)+(1.5,-.7)$)  -- ($(a)+(2.5,-.3)$)  -- ($(a)+(3.5,-.7)$)  -- ($(a)+(4.5,-.3)$)  -- ($(a)+(5.5,-.7)$)  -- ($(a)+(6.5,-.3)$)   -- ($(a)+(7.5,-.7)$)   -- 
($(a)+(8,-.5)$) -- ($(a)+(21,12.5)$) -- ($(a)+(10,12.5)$) ;
\path [fill=lightgray] ($(a)+(21,12.5)$) -- ($(a)+(10,1.5)$)  -- ($(a)+(25.5,1.5)$) 
-- ($(a)+(25.7,2)$) -- ($(a)+(25.3,2.5)$) -- ($(a)+(25.7,3)$) --   ($(a)+(25.3,3.5)$) -- ($(a)+(25.7,4)$) -- ($(a)+(25.3,4.5)$) -- ($(a)+(25.7,5)$) -- ($(a)+(25.3,5.5)$) --
($(a)+(25.5,6)$) --  ($(a)+(21.5,2)$) -- ($(a)+(21.5,12.5)$) ;
\draw [ultra thick, dashed] ($(a)+(25.5, 12.5)$) -- ($(a)+(10,12.5)$) -- ($(a)+(-.5,2)$) -- ($(a)+(-.5,1.5)$) --  ($(a)+(25.5, 1.5)$) ;

\node[color=blue] at ($(a)+(-1.5,0)$) {\tiny a=12};
\node at ($(a)+(0,0)$) {\tiny 0}; \node at ($(a)+(1,0)$) {\tiny 0}; \node at ($(a)+(2,0)$) {\tiny 0}; \node at ($(a)+(3,0)$) {\tiny 0}; \node at ($(a)+(4,0)$) {\tiny 0}; \node at ($(a)+(5,0)$) {\tiny 0}; 
\node at ($(a)+(6,0)$) {\tiny 0}; \node at ($(a)+(7,0)$) {\tiny 0}; \node at ($(a)+(8,0)$) {\tiny 0}; \node at ($(a)+(9,0)$) {\tiny 0}; \node at ($(a)+(10,0)$) {\tiny 0}; \node at ($(a)+(11,0)$) {\tiny 0}; 
\node at ($(a)+(12,0)$) {\tiny 0}; \node at ($(a)+(13,0)$) {\tiny 0}; \node at ($(a)+(14,0)$) {\tiny 0}; \node at ($(a)+(15,0)$) {\tiny 0}; \node at ($(a)+(16,0)$) {\tiny 0}; \node at ($(a)+(17,0)$) {\tiny 0}; 
\node at ($(a)+(18,0)$) {\tiny 0}; \node at ($(a)+(19,0)$) {\tiny 0}; \node at ($(a)+(20,0)$) {\tiny 0}; \node at ($(a)+(21,0)$) {\tiny 0}; \node at ($(a)+(22,0)$) {\tiny 0}; \node at ($(a)+(23,0)$) {\tiny 0}; \node at ($(a)+(24,0)$) {\tiny 0}; \node at ($(a)+(25,0)$) {\tiny 0}; 

\node[color=blue] at ($(a)+(-1.5,1)$) {\tiny a=11};
\node at ($(a)+(0,1)$) {\tiny 0}; \node at ($(a)+(1,1)$) {\tiny 0}; \node at ($(a)+(2,1)$) {\tiny 0}; \node at ($(a)+(3,1)$) {\tiny 0}; \node at ($(a)+(4,1)$) {\tiny 0}; \node at ($(a)+(5,1)$) {\tiny 0}; 
\node at ($(a)+(6,1)$) {\tiny 0}; \node at ($(a)+(7,1)$) {\tiny 0}; \node at ($(a)+(8,1)$) {\tiny 0}; \node at ($(a)+(9,1)$) {\tiny 0}; \node at ($(a)+(10,1)$) {\tiny 0}; \node at ($(a)+(11,1)$) {\tiny 0}; 
\node at ($(a)+(12,1)$) {\tiny 0}; \node at ($(a)+(13,1)$) {\tiny 0}; \node at ($(a)+(14,1)$) {\tiny 0}; \node at ($(a)+(15,1)$) {\tiny 0}; \node at ($(a)+(16,1)$) {\tiny 0}; \node at ($(a)+(17,1)$) {\tiny 0}; 
\node at ($(a)+(18,1)$) {\tiny 0}; \node at ($(a)+(19,1)$) {\tiny 0}; \node at ($(a)+(20,1)$) {\tiny 0}; \node at ($(a)+(21,1)$) {\tiny 0}; \node at ($(a)+(22,1)$) {\tiny 0}; \node at ($(a)+(23,1)$) {\tiny 0}; \node at ($(a)+(24,1)$) {\tiny 0}; \node at ($(a)+(25,1)$) {\tiny 0}; 

\node[color=blue] at ($(a)+(-1.5,2)$) {\tiny a=10};
\node at ($(a)+(0,2)$) {\tiny 1}; \node at ($(a)+(1,2)$) {\tiny 1}; \node at ($(a)+(2,2)$) {\tiny 1}; \node at ($(a)+(3,2)$) {\tiny 1}; \node at ($(a)+(4,2)$) {\tiny 1}; \node at ($(a)+(5,2)$) {\tiny 1}; 
\node at ($(a)+(6,2)$) {\tiny 1}; \node at ($(a)+(7,2)$) {\tiny 1}; \node at ($(a)+(8,2)$) {\tiny 1}; \node at ($(a)+(9,2)$) {\tiny 1}; \node at ($(a)+(10,2)$) {\tiny 1}; \node at ($(a)+(11,2)$) {\tiny 1}; 
\node at ($(a)+(12,2)$) {\tiny 1}; \node at ($(a)+(13,2)$) {\tiny 1}; \node at ($(a)+(14,2)$) {\tiny 1}; \node at ($(a)+(15,2)$) {\tiny 1}; \node at ($(a)+(16,2)$) {\tiny 1}; \node at ($(a)+(17,2)$) {\tiny 1}; 
\node at ($(a)+(18,2)$) {\tiny 1}; \node at ($(a)+(19,2)$) {\tiny 1}; \node at ($(a)+(20,2)$) {\tiny 1}; \node at ($(a)+(21,2)$) {\tiny 1}; \node at ($(a)+(22,2)$) {\tiny 1}; \node at ($(a)+(23,2)$) {\tiny 1}; \node at ($(a)+(24,2)$) {\tiny 1}; \node at ($(a)+(25,2)$) {\tiny 1}; 

\node[color=blue] at ($(a)+(-1.5,3)$) {\tiny a=9};
\node at ($(a)+(0,3)$) {\tiny 0}; \node at ($(a)+(1,3)$) {\tiny 10}; \node at ($(a)+(2,3)$) {\tiny 9}; \node at ($(a)+(3,3)$) {\tiny 8}; \node at ($(a)+(4,3)$) {\tiny 7}; \node at ($(a)+(5,3)$) {\tiny 6}; 
\node at ($(a)+(6,3)$) {\tiny 5}; \node at ($(a)+(7,3)$) {\tiny 4}; \node at ($(a)+(8,3)$) {\tiny 3}; \node at ($(a)+(9,3)$) {\tiny 2}; \node at ($(a)+(10,3)$) {\tiny 1}; \node at ($(a)+(11,3)$) {\tiny 0}; 
\node at ($(a)+(12,3)$) {\tiny 10}; \node at ($(a)+(13,3)$) {\tiny 9}; \node at ($(a)+(14,3)$) {\tiny 8}; \node at ($(a)+(15,3)$) {\tiny 7}; \node at ($(a)+(16,3)$) {\tiny 6}; \node at ($(a)+(17,3)$) {\tiny 5}; 
\node at ($(a)+(18,3)$) {\tiny 4}; \node at ($(a)+(19,3)$) {\tiny 3}; \node at ($(a)+(20,3)$) {\tiny 2}; \node at ($(a)+(21,3)$) {\tiny 1}; \node at ($(a)+(22,3)$) {\tiny 0}; \node at ($(a)+(23,3)$) {\tiny 10}; \node at ($(a)+(24,3)$) {\tiny 9}; \node at ($(a)+(25,3)$) {\tiny 8}; 

\node[color=blue] at ($(a)+(-1.5,4)$) {\tiny a=8};
\node at ($(a)+(0,4)$) {\tiny 0}; \node at ($(a)+(1,4)$) {\tiny 0}; \node at ($(a)+(2,4)$) {\tiny 1}; \node at ($(a)+(3,4)$) {\tiny 3}; \node at ($(a)+(4,4)$) {\tiny 6}; \node at ($(a)+(5,4)$) {\tiny 10}; 
\node at ($(a)+(6,4)$) {\tiny 4}; \node at ($(a)+(7,4)$) {\tiny 10}; \node at ($(a)+(8,4)$) {\tiny 6}; \node at ($(a)+(9,4)$) {\tiny 3}; \node at ($(a)+(10,4)$) {\tiny 1}; \node at ($(a)+(11,4)$) {\tiny 0}; 
\node at ($(a)+(12,4)$) {\tiny 0}; \node at ($(a)+(13,4)$) {\tiny 1}; \node at ($(a)+(14,4)$) {\tiny 3}; \node at ($(a)+(15,4)$) {\tiny 6}; \node at ($(a)+(16,4)$) {\tiny 10}; \node at ($(a)+(17,4)$) {\tiny 4}; 
\node at ($(a)+(18,4)$) {\tiny 10}; \node at ($(a)+(19,4)$) {\tiny 6}; \node at ($(a)+(20,4)$) {\tiny 3}; \node at ($(a)+(21,4)$) {\tiny 1}; \node at ($(a)+(22,4)$) {\tiny 0}; \node at ($(a)+(23,4)$) {\tiny 0}; \node at ($(a)+(24,4)$) {\tiny 1}; \node at ($(a)+(25,4)$) {\tiny 3}; 

\node[color=blue] at ($(a)+(-1.5,5)$) {\tiny a=7};
\node at ($(a)+(0,5)$) {\tiny 0}; \node at ($(a)+(1,5)$) {\tiny 0}; \node at ($(a)+(2,5)$) {\tiny 0}; \node at ($(a)+(3,5)$) {\tiny 10}; \node at ($(a)+(4,5)$) {\tiny 7}; \node at ($(a)+(5,5)$) {\tiny 1}; 
\node at ($(a)+(6,5)$) {\tiny 2}; \node at ($(a)+(7,5)$) {\tiny 9}; \node at ($(a)+(8,5)$) {\tiny 10}; \node at ($(a)+(9,5)$) {\tiny 4}; \node at ($(a)+(10,5)$) {\tiny 1}; \node at ($(a)+(11,5)$) {\tiny 0}; 
\node at ($(a)+(12,5)$) {\tiny 0}; \node at ($(a)+(13,5)$) {\tiny 0}; \node at ($(a)+(14,5)$) {\tiny 10}; \node at ($(a)+(15,5)$) {\tiny 7}; \node at ($(a)+(16,5)$) {\tiny 1}; \node at ($(a)+(17,5)$) {\tiny 2}; 
\node at ($(a)+(18,5)$) {\tiny 9}; \node at ($(a)+(19,5)$) {\tiny 10}; \node at ($(a)+(20,5)$) {\tiny 4}; \node at ($(a)+(21,5)$) {\tiny 1}; \node at ($(a)+(22,5)$) {\tiny 0}; \node at ($(a)+(23,5)$) {\tiny 0}; \node at ($(a)+(24,5)$) {\tiny 0}; \node at ($(a)+(25,5)$) {\tiny 10}; 

\node[color=blue] at ($(a)+(-1.5,6)$) {\tiny a=6};
\node at ($(a)+(0,6)$) {\tiny 0}; \node at ($(a)+(1,6)$) {\tiny 0}; \node at ($(a)+(2,6)$) {\tiny 0}; \node at ($(a)+(3,6)$) {\tiny 0}; \node at ($(a)+(4,6)$) {\tiny 1}; \node at ($(a)+(5,6)$) {\tiny 5}; 
\node at ($(a)+(6,6)$) {\tiny 4}; \node at ($(a)+(7,6)$) {\tiny 2}; \node at ($(a)+(8,6)$) {\tiny 4}; \node at ($(a)+(9,6)$) {\tiny 5}; \node at ($(a)+(10,6)$) {\tiny 1}; \node at ($(a)+(11,6)$) {\tiny 0}; 
\node at ($(a)+(12,6)$) {\tiny 0}; \node at ($(a)+(13,6)$) {\tiny 0}; \node at ($(a)+(14,6)$) {\tiny 0}; \node at ($(a)+(15,6)$) {\tiny 1}; \node at ($(a)+(16,6)$) {\tiny 5}; \node at ($(a)+(17,6)$) {\tiny 4}; 
\node at ($(a)+(18,6)$) {\tiny 2}; \node at ($(a)+(19,6)$) {\tiny 4}; \node at ($(a)+(20,6)$) {\tiny 5}; \node at ($(a)+(21,6)$) {\tiny 1}; \node at ($(a)+(22,6)$) {\tiny 0}; \node at ($(a)+(23,6)$) {\tiny 0}; \node at ($(a)+(24,6)$) {\tiny 0}; \node at ($(a)+(25,6)$) {\tiny 0}; 

\node[color=blue] at ($(a)+(-1.5,7)$) {\tiny a=5};
\node at ($(a)+(0,7)$) {\tiny 0}; \node at ($(a)+(1,7)$) {\tiny 0}; \node at ($(a)+(2,7)$) {\tiny 0}; \node at ($(a)+(3,7)$) {\tiny 0}; \node at ($(a)+(4,7)$) {\tiny 0}; \node at ($(a)+(5,7)$) {\tiny 10}; 
\node at ($(a)+(6,7)$) {\tiny 5}; \node at ($(a)+(7,7)$) {\tiny 1}; \node at ($(a)+(8,7)$) {\tiny 10}; \node at ($(a)+(9,7)$) {\tiny 6}; \node at ($(a)+(10,7)$) {\tiny 1}; \node at ($(a)+(11,7)$) {\tiny 0}; 
\node at ($(a)+(12,7)$) {\tiny 0}; \node at ($(a)+(13,7)$) {\tiny 0}; \node at ($(a)+(14,7)$) {\tiny 0}; \node at ($(a)+(15,7)$) {\tiny 0}; \node at ($(a)+(16,7)$) {\tiny 10}; \node at ($(a)+(17,7)$) {\tiny 5}; 
\node at ($(a)+(18,7)$) {\tiny 1}; \node at ($(a)+(19,7)$) {\tiny 10}; \node at ($(a)+(20,7)$) {\tiny 6}; \node at ($(a)+(21,7)$) {\tiny 1}; \node at ($(a)+(22,7)$) {\tiny 0}; \node at ($(a)+(23,7)$) {\tiny 0}; \node at ($(a)+(24,7)$) {\tiny 0}; \node at ($(a)+(25,7)$) {\tiny 0}; 

\node[color=blue] at ($(a)+(-1.5,8)$) {\tiny a=4};
\node at ($(a)+(0,8)$) {\tiny 0}; \node at ($(a)+(1,8)$) {\tiny 0}; \node at ($(a)+(2,8)$) {\tiny 0}; \node at ($(a)+(3,8)$) {\tiny 0}; \node at ($(a)+(4,8)$) {\tiny 0}; \node at ($(a)+(5,8)$) {\tiny 0}; 
\node at ($(a)+(6,8)$) {\tiny 1}; \node at ($(a)+(7,8)$) {\tiny 7}; \node at ($(a)+(8,8)$) {\tiny 6}; \node at ($(a)+(9,8)$) {\tiny 7}; \node at ($(a)+(10,8)$) {\tiny 1}; \node at ($(a)+(11,8)$) {\tiny 0}; 
\node at ($(a)+(12,8)$) {\tiny 0}; \node at ($(a)+(13,8)$) {\tiny 0}; \node at ($(a)+(14,8)$) {\tiny 0}; \node at ($(a)+(15,8)$) {\tiny 0}; \node at ($(a)+(16,8)$) {\tiny 0}; \node at ($(a)+(17,8)$) {\tiny 1}; 
\node at ($(a)+(18,8)$) {\tiny 7}; \node at ($(a)+(19,8)$) {\tiny 6}; \node at ($(a)+(20,8)$) {\tiny 7}; \node at ($(a)+(21,8)$) {\tiny 1}; \node at ($(a)+(22,8)$) {\tiny 0}; \node at ($(a)+(23,8)$) {\tiny 0}; \node at ($(a)+(24,8)$) {\tiny 0};  \node at ($(a)+(25,8)$) {\tiny 0}; 

\node[color=blue] at ($(a)+(-1.5,9)$) {\tiny a=3};
\node at ($(a)+(0,9)$) {\tiny 0}; \node at ($(a)+(1,9)$) {\tiny 0}; \node at ($(a)+(2,9)$) {\tiny 0}; \node at ($(a)+(3,9)$) {\tiny 0}; \node at ($(a)+(4,9)$) {\tiny 0}; \node at ($(a)+(5,9)$) {\tiny 0}; 
\node at ($(a)+(6,9)$) {\tiny 0}; \node at ($(a)+(7,9)$) {\tiny 10}; \node at ($(a)+(8,9)$) {\tiny 3}; \node at ($(a)+(9,9)$) {\tiny 8}; \node at ($(a)+(10,9)$) {\tiny 1}; \node at ($(a)+(11,9)$) {\tiny 0}; 
\node at ($(a)+(12,9)$) {\tiny 0}; \node at ($(a)+(13,9)$) {\tiny 0}; \node at ($(a)+(14,9)$) {\tiny 0}; \node at ($(a)+(15,9)$) {\tiny 0}; \node at ($(a)+(16,9)$) {\tiny 0}; \node at ($(a)+(17,9)$) {\tiny 0}; 
\node at ($(a)+(18,9)$) {\tiny 10}; \node at ($(a)+(19,9)$) {\tiny 3}; \node at ($(a)+(20,9)$) {\tiny 8}; \node at ($(a)+(21,9)$) {\tiny 1}; \node at ($(a)+(22,9)$) {\tiny 0}; \node at ($(a)+(23,9)$) {\tiny 0}; \node at ($(a)+(24,9)$) {\tiny 0};  \node at ($(a)+(25,9)$) {\tiny 0}; 

\node[color=blue] at ($(a)+(-1.5,10)$) {\tiny a=2};
\node at ($(a)+(0,10)$) {\tiny 0}; \node at ($(a)+(1,10)$) {\tiny 0}; \node at ($(a)+(2,10)$) {\tiny 0}; \node at ($(a)+(3,10)$) {\tiny 0}; \node at ($(a)+(4,10)$) {\tiny 0}; \node at ($(a)+(5,10)$) {\tiny 0}; 
\node at ($(a)+(6,10)$) {\tiny 0}; \node at ($(a)+(7,10)$) {\tiny 0}; \node at ($(a)+(8,10)$) {\tiny 1}; \node at ($(a)+(9,10)$) {\tiny 9}; \node at ($(a)+(10,10)$) {\tiny 1}; \node at ($(a)+(11,10)$) {\tiny 0}; 
\node at ($(a)+(12,10)$) {\tiny 0}; \node at ($(a)+(13,10)$) {\tiny 0}; \node at ($(a)+(14,10)$) {\tiny 0}; \node at ($(a)+(15,10)$) {\tiny 0}; \node at ($(a)+(16,10)$) {\tiny 0}; \node at ($(a)+(17,10)$) {\tiny 0}; 
\node at ($(a)+(18,10)$) {\tiny 0}; \node at ($(a)+(19,10)$) {\tiny 1}; \node at ($(a)+(20,10)$) {\tiny 9}; \node at ($(a)+(21,10)$) {\tiny 1}; \node at ($(a)+(22,10)$) {\tiny 0}; \node at ($(a)+(23,10)$) {\tiny 0}; \node at ($(a)+(24,10)$) {\tiny 0};  \node at ($(a)+(25,10)$) {\tiny 0}; 

\node[color=blue] at ($(a)+(-1.5,11)$) {\tiny a=1};
\node at ($(a)+(0,11)$) {\tiny 0}; \node at ($(a)+(1,11)$) {\tiny 0}; \node at ($(a)+(2,11)$) {\tiny 0}; \node at ($(a)+(3,11)$) {\tiny 0}; \node at ($(a)+(4,11)$) {\tiny 0}; \node at ($(a)+(5,11)$) {\tiny 0}; 
\node at ($(a)+(6,11)$) {\tiny 0}; \node at ($(a)+(7,11)$) {\tiny 0}; \node at ($(a)+(8,11)$) {\tiny 0}; \node at ($(a)+(9,11)$) {\tiny 10}; \node at ($(a)+(10,11)$) {\tiny 1}; \node at ($(a)+(11,11)$) {\tiny 0}; 
\node at ($(a)+(12,11)$) {\tiny 0}; \node at ($(a)+(13,11)$) {\tiny 0}; \node at ($(a)+(14,11)$) {\tiny 0}; \node at ($(a)+(15,11)$) {\tiny 0}; \node at ($(a)+(16,11)$) {\tiny 0}; \node at ($(a)+(17,11)$) {\tiny 0}; 
\node at ($(a)+(18,11)$) {\tiny 0}; \node at ($(a)+(19,11)$) {\tiny 0}; \node at ($(a)+(20,11)$) {\tiny 10}; \node at ($(a)+(21,11)$) {\tiny 1}; \node at ($(a)+(22,11)$) {\tiny 0}; \node at ($(a)+(23,11)$) {\tiny 0}; \node at ($(a)+(24,11)$) {\tiny 0};  \node at ($(a)+(25,11)$) {\tiny 0}; 

\node[color=blue] at ($(a)+(-1.5,12)$) {\tiny a=0};
\node at ($(a)+(0,12)$) {\tiny 0}; \node at ($(a)+(1,12)$) {\tiny 0}; \node at ($(a)+(2,12)$) {\tiny 0}; \node at ($(a)+(3,12)$) {\tiny 0}; \node at ($(a)+(4,12)$) {\tiny 0}; \node at ($(a)+(5,12)$) {\tiny 0}; 
\node at ($(a)+(6,12)$) {\tiny 0}; \node at ($(a)+(7,12)$) {\tiny 0}; \node at ($(a)+(8,12)$) {\tiny 0}; \node at ($(a)+(9,12)$) {\tiny 0}; \node at ($(a)+(10,12)$) {\tiny 1}; \node at ($(a)+(11,12)$) {\tiny 0}; 
\node at ($(a)+(12,12)$) {\tiny 0}; \node at ($(a)+(13,12)$) {\tiny 0}; \node at ($(a)+(14,12)$) {\tiny 0}; \node at ($(a)+(15,12)$) {\tiny 0}; \node at ($(a)+(16,12)$) {\tiny 0}; \node at ($(a)+(17,12)$) {\tiny 0}; 
\node at ($(a)+(18,12)$) {\tiny 0}; \node at ($(a)+(19,12)$) {\tiny 0}; \node at ($(a)+(20,12)$) {\tiny 0}; \node at ($(a)+(21,12)$) {\tiny 1}; \node at ($(a)+(22,12)$) {\tiny 0}; \node at ($(a)+(23,12)$) {\tiny 0}; \node at ($(a)+(24,12)$) {\tiny 0};  \node at ($(a)+(25,12)$) {\tiny 0}; 

\node[color=blue] at ($(a)+(-0.41,13.5)$) {\tiny b=0}; 
\node[color=blue] at ($(a)+(1,13.5)$) {\tiny 1}; \node[color=blue] at ($(a)+(2,13.5)$) {\tiny 2}; 
\node[color=blue] at ($(a)+(3,13.5)$) {\tiny 3}; \node[color=blue] at ($(a)+(4,13.5)$) {\tiny 4}; \node[color=blue] at ($(a)+(5,13.5)$) {\tiny 5}; 
\node[color=blue] at ($(a)+(6,13.5)$) {\tiny 6}; \node[color=blue] at ($(a)+(7,13.5)$) {\tiny 7}; \node[color=blue] at ($(a)+(8,13.5)$) {\tiny 8}; 
\node[color=blue] at ($(a)+(9,13.5)$) {\tiny 9}; \node[color=blue] at ($(a)+(10,13.5)$) {\tiny 10}; \node[color=blue] at ($(a)+(11,13.5)$) {\tiny 11}; 
\node[color=blue] at ($(a)+(12,13.5)$) {\tiny 12}; \node[color=blue] at ($(a)+(13,13.5)$) {\tiny 13}; \node[color=blue] at ($(a)+(14,13.5)$) {\tiny 14}; 
\node[color=blue] at ($(a)+(15,13.5)$) {\tiny 15}; \node[color=blue] at ($(a)+(16,13.5)$) {\tiny 16}; \node[color=blue] at ($(a)+(17,13.5)$) {\tiny 17}; 
\node[color=blue] at ($(a)+(18,13.5)$) {\tiny 18}; \node[color=blue] at ($(a)+(19,13.5)$) {\tiny 19}; \node[color=blue] at ($(a)+(20,13.5)$) {\tiny 20}; 
\node[color=blue] at ($(a)+(21,13.5)$) {\tiny 21}; \node[color=blue] at ($(a)+(22,13.5)$) {\tiny 22}; \node[color=blue] at ($(a)+(23,13.5)$) {\tiny 23}; 
\node[color=blue] at ($(a)+(24,13.5)$) {\tiny 24}; \node[color=blue] at ($(a)+(25,13.5)$) {\tiny 25};


\coordinate (a) at (0,-11.5);

\draw ($ (a)+ (-2.2,10.3) $) -- ($ (a)+(25.5,10.3) $);

\path [fill=yellow] ($(a)+(7,9.5)$) -- ($(a)+(-.5,2)$)  -- ($(a)+(-.5,-.5)$)  --
($(a)+(.5,-.7)$)  --($(a)+(1.5,-.3)$)  --($(a)+(2.5,-.7)$)  --($(a)+(3.5,-.3)$)  --($(a)+(4.5,-.7)$)   --
 ($(a)+(5,-.5)$) -- ($(a)+(15,9.5)$) -- ($(a)+(7,9.5)$) ;
\path [fill=lightgray] ($(a)+(18,9.5)$) -- ($(a)+(10.5,2)$)  -- ($(a)+(10.5,5)$) -- ($(a)+(7,1.5)$) -- ($(a)+(25.5,1.5)$) 
-- ($(a)+(25.7,2)$) -- ($(a)+(25.3,2.5)$) -- ($(a)+(25.7,3)$) -- ($(a)+(25.3,3.5)$) -- ($(a)+(25.7,4)$) -- ($(a)+(25.3,4.5)$)  -- ($(a)+(25.7,5)$)  -- ($(a)+(25.3,5.5)$) 
-- ($(a)+(25.5,6)$)  -- ($(a)+(21.5,2)$) -- ($(a)+(21.5,5)$)  -- ($(a)+(18.5,2)$) -- ($(a)+(18.5,9.5)$);
\draw [ultra thick, dashed] ($(a)+(25.5, 9.5)$) -- ($(a)+(7,9.5)$) -- ($(a)+(-.5,2)$) -- ($(a)+(-.5,1.5)$) --  ($(a)+(25.5, 1.5)$) ;

\node[color=blue] at ($(a)+(-1.5,0)$) {\tiny a=9};
\node at ($(a)+(0,0)$) {\tiny 0}; \node at ($(a)+(1,0)$) {\tiny 0}; \node at ($(a)+(2,0)$) {\tiny 0}; \node at ($(a)+(3,0)$) {\tiny 0}; \node at ($(a)+(4,0)$) {\tiny 0}; \node at ($(a)+(5,0)$) {\tiny 0}; 
\node at ($(a)+(6,0)$) {\tiny 0}; \node at ($(a)+(7,0)$) {\tiny 0}; \node at ($(a)+(8,0)$) {\tiny 0}; \node at ($(a)+(9,0)$) {\tiny 0}; \node at ($(a)+(10,0)$) {\tiny 0}; \node at ($(a)+(11,0)$) {\tiny 0}; 
\node at ($(a)+(12,0)$) {\tiny 0}; \node at ($(a)+(13,0)$) {\tiny 0}; \node at ($(a)+(14,0)$) {\tiny 0}; \node at ($(a)+(15,0)$) {\tiny 0}; \node at ($(a)+(16,0)$) {\tiny 0}; \node at ($(a)+(17,0)$) {\tiny 0}; 
\node at ($(a)+(18,0)$) {\tiny 0}; \node at ($(a)+(19,0)$) {\tiny 0}; \node at ($(a)+(20,0)$) {\tiny 0}; \node at ($(a)+(21,0)$) {\tiny 0}; \node at ($(a)+(22,0)$) {\tiny 0}; \node at ($(a)+(23,0)$) {\tiny 0}; \node at ($(a)+(24,0)$) {\tiny 0}; \node at ($(a)+(25,0)$) {\tiny 0}; 

\node[color=blue] at ($(a)+(-1.5,1)$) {\tiny a=8};
\node at ($(a)+(0,1)$) {\tiny 0}; \node at ($(a)+(1,1)$) {\tiny 0}; \node at ($(a)+(2,1)$) {\tiny 0}; \node at ($(a)+(3,1)$) {\tiny 0}; \node at ($(a)+(4,1)$) {\tiny 0}; \node at ($(a)+(5,1)$) {\tiny 0}; 
\node at ($(a)+(6,1)$) {\tiny 0}; \node at ($(a)+(7,1)$) {\tiny 0}; \node at ($(a)+(8,1)$) {\tiny 0}; \node at ($(a)+(9,1)$) {\tiny 0}; \node at ($(a)+(10,1)$) {\tiny 0}; \node at ($(a)+(11,1)$) {\tiny 0}; 
\node at ($(a)+(12,1)$) {\tiny 0}; \node at ($(a)+(13,1)$) {\tiny 0}; \node at ($(a)+(14,1)$) {\tiny 0}; \node at ($(a)+(15,1)$) {\tiny 0}; \node at ($(a)+(16,1)$) {\tiny 0}; \node at ($(a)+(17,1)$) {\tiny 0}; 
\node at ($(a)+(18,1)$) {\tiny 0}; \node at ($(a)+(19,1)$) {\tiny 0}; \node at ($(a)+(20,1)$) {\tiny 0}; \node at ($(a)+(21,1)$) {\tiny 0}; \node at ($(a)+(22,1)$) {\tiny 0}; \node at ($(a)+(23,1)$) {\tiny 0}; \node at ($(a)+(24,1)$) {\tiny 0}; \node at ($(a)+(25,1)$) {\tiny 0}; 

\node[color=blue] at ($(a)+(-1.5,2)$) {\tiny a=7};
\node at ($(a)+(0,2)$) {\tiny 2}; \node at ($(a)+(1,2)$) {\tiny 2}; \node at ($(a)+(2,2)$) {\tiny 2}; \node at ($(a)+(3,2)$) {\tiny 2}; \node at ($(a)+(4,2)$) {\tiny 2}; \node at ($(a)+(5,2)$) {\tiny 2}; 
\node at ($(a)+(6,2)$) {\tiny 2}; \node at ($(a)+(7,2)$) {\tiny 2}; \node at ($(a)+(8,2)$) {\tiny 2}; \node at ($(a)+(9,2)$) {\tiny 2}; \node at ($(a)+(10,2)$) {\tiny 2}; \node at ($(a)+(11,2)$) {\tiny 2}; 
\node at ($(a)+(12,2)$) {\tiny 2}; \node at ($(a)+(13,2)$) {\tiny 2}; \node at ($(a)+(14,2)$) {\tiny 2}; \node at ($(a)+(15,2)$) {\tiny 2}; \node at ($(a)+(16,2)$) {\tiny 2}; \node at ($(a)+(17,2)$) {\tiny 2}; 
\node at ($(a)+(18,2)$) {\tiny 2}; \node at ($(a)+(19,2)$) {\tiny 2}; \node at ($(a)+(20,2)$) {\tiny 2}; \node at ($(a)+(21,2)$) {\tiny 2}; \node at ($(a)+(22,2)$) {\tiny 2}; \node at ($(a)+(23,2)$) {\tiny 2}; \node at ($(a)+(24,2)$) {\tiny 2}; \node at ($(a)+(25,2)$) {\tiny 2}; 

\node[color=blue] at ($(a)+(-1.5,3)$) {\tiny a=6};
\node at ($(a)+(0,3)$) {\tiny 0}; \node at ($(a)+(1,3)$) {\tiny 2}; \node at ($(a)+(2,3)$) {\tiny 5}; \node at ($(a)+(3,3)$) {\tiny 9}; \node at ($(a)+(4,3)$) {\tiny 3}; \node at ($(a)+(5,3)$) {\tiny 9}; 
\node at ($(a)+(6,3)$) {\tiny 5}; \node at ($(a)+(7,3)$) {\tiny 2}; \node at ($(a)+(8,3)$) {\tiny 0}; \node at ($(a)+(9,3)$) {\tiny 10}; \node at ($(a)+(10,3)$) {\tiny 10}; \node at ($(a)+(11,3)$) {\tiny 0}; 
\node at ($(a)+(12,3)$) {\tiny 2}; \node at ($(a)+(13,3)$) {\tiny 5}; \node at ($(a)+(14,3)$) {\tiny 9}; \node at ($(a)+(15,3)$) {\tiny 3}; \node at ($(a)+(16,3)$) {\tiny 9}; \node at ($(a)+(17,3)$) {\tiny 5}; 
\node at ($(a)+(18,3)$) {\tiny 2}; \node at ($(a)+(19,3)$) {\tiny 0}; \node at ($(a)+(20,3)$) {\tiny 10}; \node at ($(a)+(21,3)$) {\tiny 10}; \node at ($(a)+(22,3)$) {\tiny 0}; \node at ($(a)+(23,3)$) {\tiny 2}; \node at ($(a)+(24,3)$) {\tiny 5}; \node at ($(a)+(25,3)$) {\tiny 9}; 

\node[color=blue] at ($(a)+(-1.5,4)$) {\tiny a=5};
\node at ($(a)+(0,4)$) {\tiny 0}; \node at ($(a)+(1,4)$) {\tiny 0}; \node at ($(a)+(2,4)$) {\tiny 2}; \node at ($(a)+(3,4)$) {\tiny 9}; \node at ($(a)+(4,4)$) {\tiny 1}; \node at ($(a)+(5,4)$) {\tiny 1}; 
\node at ($(a)+(6,4)$) {\tiny 9}; \node at ($(a)+(7,4)$) {\tiny 2}; \node at ($(a)+(8,4)$) {\tiny 0}; \node at ($(a)+(9,4)$) {\tiny 0}; \node at ($(a)+(10,4)$) {\tiny 9}; \node at ($(a)+(11,4)$) {\tiny 0}; 
\node at ($(a)+(12,4)$) {\tiny 0}; \node at ($(a)+(13,4)$) {\tiny 2}; \node at ($(a)+(14,4)$) {\tiny 9}; \node at ($(a)+(15,4)$) {\tiny 1}; \node at ($(a)+(16,4)$) {\tiny 1}; \node at ($(a)+(17,4)$) {\tiny 9}; 
\node at ($(a)+(18,4)$) {\tiny 2}; \node at ($(a)+(19,4)$) {\tiny 0}; \node at ($(a)+(20,4)$) {\tiny 0}; \node at ($(a)+(21,4)$) {\tiny 9}; \node at ($(a)+(22,4)$) {\tiny 0}; \node at ($(a)+(23,4)$) {\tiny 0}; \node at ($(a)+(24,4)$) {\tiny 2}; \node at ($(a)+(25,4)$) {\tiny 9}; 

\node[color=blue] at ($(a)+(-1.5,5)$) {\tiny a=4};
\node at ($(a)+(0,5)$) {\tiny 0}; \node at ($(a)+(1,5)$) {\tiny 0}; \node at ($(a)+(2,5)$) {\tiny 0}; \node at ($(a)+(3,5)$) {\tiny 2}; \node at ($(a)+(4,5)$) {\tiny 3}; \node at ($(a)+(5,5)$) {\tiny 1}; 
\node at ($(a)+(6,5)$) {\tiny 3}; \node at ($(a)+(7,5)$) {\tiny 2}; \node at ($(a)+(8,5)$) {\tiny 0}; \node at ($(a)+(9,5)$) {\tiny 0}; \node at ($(a)+(10,5)$) {\tiny 0}; \node at ($(a)+(11,5)$) {\tiny 0}; 
\node at ($(a)+(12,5)$) {\tiny 0}; \node at ($(a)+(13,5)$) {\tiny 0}; \node at ($(a)+(14,5)$) {\tiny 2}; \node at ($(a)+(15,5)$) {\tiny 3}; \node at ($(a)+(16,5)$) {\tiny 1}; \node at ($(a)+(17,5)$) {\tiny 3}; 
\node at ($(a)+(18,5)$) {\tiny 2}; \node at ($(a)+(19,5)$) {\tiny 0}; \node at ($(a)+(20,5)$) {\tiny 0}; \node at ($(a)+(21,5)$) {\tiny 0}; \node at ($(a)+(22,5)$) {\tiny 0}; \node at ($(a)+(23,5)$) {\tiny 0}; \node at ($(a)+(24,5)$) {\tiny 0}; \node at ($(a)+(25,5)$) {\tiny 2}; 

\node[color=blue] at ($(a)+(-1.5,6)$) {\tiny a=3};
\node at ($(a)+(0,6)$) {\tiny 0}; \node at ($(a)+(1,6)$) {\tiny 0}; \node at ($(a)+(2,6)$) {\tiny 0}; \node at ($(a)+(3,6)$) {\tiny 0}; \node at ($(a)+(4,6)$) {\tiny 2}; \node at ($(a)+(5,6)$) {\tiny 9}; 
\node at ($(a)+(6,6)$) {\tiny 9}; \node at ($(a)+(7,6)$) {\tiny 2}; \node at ($(a)+(8,6)$) {\tiny 0}; \node at ($(a)+(9,6)$) {\tiny 0}; \node at ($(a)+(10,6)$) {\tiny 0}; \node at ($(a)+(11,6)$) {\tiny 0}; 
\node at ($(a)+(12,6)$) {\tiny 0}; \node at ($(a)+(13,6)$) {\tiny 0}; \node at ($(a)+(14,6)$) {\tiny 0}; \node at ($(a)+(15,6)$) {\tiny 2}; \node at ($(a)+(16,6)$) {\tiny 9}; \node at ($(a)+(17,6)$) {\tiny 9}; 
\node at ($(a)+(18,6)$) {\tiny 2}; \node at ($(a)+(19,6)$) {\tiny 0}; \node at ($(a)+(20,6)$) {\tiny 0}; \node at ($(a)+(21,6)$) {\tiny 0}; \node at ($(a)+(22,6)$) {\tiny 0}; \node at ($(a)+(23,6)$) {\tiny 0}; \node at ($(a)+(24,6)$) {\tiny 0}; \node at ($(a)+(25,6)$) {\tiny 0}; 

\node[color=blue] at ($(a)+(-1.5,7)$) {\tiny a=2};
\node at ($(a)+(0,7)$) {\tiny 0}; \node at ($(a)+(1,7)$) {\tiny 0}; \node at ($(a)+(2,7)$) {\tiny 0}; \node at ($(a)+(3,7)$) {\tiny 0}; \node at ($(a)+(4,7)$) {\tiny 0}; \node at ($(a)+(5,7)$) {\tiny 2}; 
\node at ($(a)+(6,7)$) {\tiny 5}; \node at ($(a)+(7,7)$) {\tiny 2}; \node at ($(a)+(8,7)$) {\tiny 0}; \node at ($(a)+(9,7)$) {\tiny 0}; \node at ($(a)+(10,7)$) {\tiny 0}; \node at ($(a)+(11,7)$) {\tiny 0}; 
\node at ($(a)+(12,7)$) {\tiny 0}; \node at ($(a)+(13,7)$) {\tiny 0}; \node at ($(a)+(14,7)$) {\tiny 0}; \node at ($(a)+(15,7)$) {\tiny 0}; \node at ($(a)+(16,7)$) {\tiny 2}; \node at ($(a)+(17,7)$) {\tiny 5}; 
\node at ($(a)+(18,7)$) {\tiny 2}; \node at ($(a)+(19,7)$) {\tiny 0}; \node at ($(a)+(20,7)$) {\tiny 0}; \node at ($(a)+(21,7)$) {\tiny 0}; \node at ($(a)+(22,7)$) {\tiny 0}; \node at ($(a)+(23,7)$) {\tiny 0}; \node at ($(a)+(24,7)$) {\tiny 0}; \node at ($(a)+(25,7)$) {\tiny 0}; 

\node[color=blue] at ($(a)+(-1.5,8)$) {\tiny a=1};
\node at ($(a)+(0,8)$) {\tiny 0}; \node at ($(a)+(1,8)$) {\tiny 0}; \node at ($(a)+(2,8)$) {\tiny 0}; \node at ($(a)+(3,8)$) {\tiny 0}; \node at ($(a)+(4,8)$) {\tiny 0}; \node at ($(a)+(5,8)$) {\tiny 0}; 
\node at ($(a)+(6,8)$) {\tiny 2}; \node at ($(a)+(7,8)$) {\tiny 2}; \node at ($(a)+(8,8)$) {\tiny 0}; \node at ($(a)+(9,8)$) {\tiny 0}; \node at ($(a)+(10,8)$) {\tiny 0}; \node at ($(a)+(11,8)$) {\tiny 0}; 
\node at ($(a)+(12,8)$) {\tiny 0}; \node at ($(a)+(13,8)$) {\tiny 0}; \node at ($(a)+(14,8)$) {\tiny 0}; \node at ($(a)+(15,8)$) {\tiny 0}; \node at ($(a)+(16,8)$) {\tiny 0}; \node at ($(a)+(17,8)$) {\tiny 2}; 
\node at ($(a)+(18,8)$) {\tiny 2}; \node at ($(a)+(19,8)$) {\tiny 0}; \node at ($(a)+(20,8)$) {\tiny 0}; \node at ($(a)+(21,8)$) {\tiny 0}; \node at ($(a)+(22,8)$) {\tiny 0}; \node at ($(a)+(23,8)$) {\tiny 0}; \node at ($(a)+(24,8)$) {\tiny 0};  \node at ($(a)+(25,8)$) {\tiny 0}; 

\node[color=blue] at ($(a)+(-1.5,9)$) {\tiny a=0};
\node at ($(a)+(0,9)$) {\tiny 0}; \node at ($(a)+(1,9)$) {\tiny 0}; \node at ($(a)+(2,9)$) {\tiny 0}; \node at ($(a)+(3,9)$) {\tiny 0}; \node at ($(a)+(4,9)$) {\tiny 0}; \node at ($(a)+(5,9)$) {\tiny 0}; 
\node at ($(a)+(6,9)$) {\tiny 0}; \node at ($(a)+(7,9)$) {\tiny 2}; \node at ($(a)+(8,9)$) {\tiny 0}; \node at ($(a)+(9,9)$) {\tiny 0}; \node at ($(a)+(10,9)$) {\tiny 0}; \node at ($(a)+(11,9)$) {\tiny 0}; 
\node at ($(a)+(12,9)$) {\tiny 0}; \node at ($(a)+(13,9)$) {\tiny 0}; \node at ($(a)+(14,9)$) {\tiny 0}; \node at ($(a)+(15,9)$) {\tiny 0}; \node at ($(a)+(16,9)$) {\tiny 0}; \node at ($(a)+(17,9)$) {\tiny 0}; 
\node at ($(a)+(18,9)$) {\tiny 2}; \node at ($(a)+(19,9)$) {\tiny 0}; \node at ($(a)+(20,9)$) {\tiny 0}; \node at ($(a)+(21,9)$) {\tiny 0}; \node at ($(a)+(22,9)$) {\tiny 0}; \node at ($(a)+(23,9)$) {\tiny 0}; \node at ($(a)+(24,9)$) {\tiny 0};  \node at ($(a)+(25,9)$) {\tiny 0}; 


\coordinate (a) at (0,-22);

\draw ($ (a)+ (-2.2,9.3) $) -- ($ (a)+(25.5,9.3) $);

\path [fill=yellow] ($(a)+(4,8.5)$) -- ($(a)+(-.5,4)$)  -- ($(a)+(-.5,-1)$) -- ($(a)+(9,8.5)$) -- ($(a)+(4,8.5)$) ;
\path [fill=lightgray] ($(a)+(7.5,7)$) -- ($(a)+(4,3.5)$)  -- ($(a)+(25.5,3.5)$)  
-- ($(a)+(25.7,4)$) -- ($(a)+(25.3,4.5)$) -- ($(a)+(25.7,5)$) -- ($(a)+(25.3,5.5)$) -- ($(a)+(25.7,6)$) -- ($(a)+(25.3,6.5)$) -- ($(a)+(25.7,7)$) -- ($(a)+(25.3,7.5)$) 
-- ($(a)+(25.5,8)$) -- ($(a)+(21.5,4)$) -- ($(a)+(21.5,7)$)  -- ($(a)+(18.5,4)$)  -- ($(a)+(18.5,7)$) -- ($(a)+(15.5,4)$) --  ($(a)+(15.5,8.5)$) -- ($(a)+(15,8.5)$) -- ($(a)+(10.5,4)$) -- ($(a)+(10.5,7)$) -- ($(a)+(7.5,4)$) ;
\draw [ultra thick, dashed] ($(a)+(25.5, 8.5)$) -- ($(a)+(4,8.5)$) -- ($(a)+(-.5,4)$) -- ($(a)+(-.5,3.5)$) --  ($(a)+(25.5, 3.5)$) ;

\node[color=blue] at ($(a)+(-1.5,0)$) {\tiny a=8};
\node at ($(a)+(0,0)$) {\tiny 0}; \node at ($(a)+(1,0)$) {\tiny 0}; \node at ($(a)+(2,0)$) {\tiny 0}; \node at ($(a)+(3,0)$) {\tiny 0}; \node at ($(a)+(4,0)$) {\tiny 0}; \node at ($(a)+(5,0)$) {\tiny 0}; 
\node at ($(a)+(6,0)$) {\tiny 0}; \node at ($(a)+(7,0)$) {\tiny 0}; \node at ($(a)+(8,0)$) {\tiny 0}; \node at ($(a)+(9,0)$) {\tiny 0}; \node at ($(a)+(10,0)$) {\tiny 0}; \node at ($(a)+(11,0)$) {\tiny 0}; 
\node at ($(a)+(12,0)$) {\tiny 0}; \node at ($(a)+(13,0)$) {\tiny 0}; \node at ($(a)+(14,0)$) {\tiny 0}; \node at ($(a)+(15,0)$) {\tiny 0}; \node at ($(a)+(16,0)$) {\tiny 0}; \node at ($(a)+(17,0)$) {\tiny 0}; 
\node at ($(a)+(18,0)$) {\tiny 0}; \node at ($(a)+(19,0)$) {\tiny 0}; \node at ($(a)+(20,0)$) {\tiny 0}; \node at ($(a)+(21,0)$) {\tiny 0}; \node at ($(a)+(22,0)$) {\tiny 0}; \node at ($(a)+(23,0)$) {\tiny 0}; \node at ($(a)+(24,0)$) {\tiny 0}; \node at ($(a)+(25,0)$) {\tiny 0};

\node[color=blue] at ($(a)+(-1.5,1)$) {\tiny a=7};
\node at ($(a)+(0,1)$) {\tiny 0}; \node at ($(a)+(1,1)$) {\tiny 0}; \node at ($(a)+(2,1)$) {\tiny 0}; \node at ($(a)+(3,1)$) {\tiny 0}; \node at ($(a)+(4,1)$) {\tiny 0}; \node at ($(a)+(5,1)$) {\tiny 0}; 
\node at ($(a)+(6,1)$) {\tiny 0}; \node at ($(a)+(7,1)$) {\tiny 0}; \node at ($(a)+(8,1)$) {\tiny 0}; \node at ($(a)+(9,1)$) {\tiny 0}; \node at ($(a)+(10,1)$) {\tiny 0}; \node at ($(a)+(11,1)$) {\tiny 0}; 
\node at ($(a)+(12,1)$) {\tiny 0}; \node at ($(a)+(13,1)$) {\tiny 0}; \node at ($(a)+(14,1)$) {\tiny 0}; \node at ($(a)+(15,1)$) {\tiny 0}; \node at ($(a)+(16,1)$) {\tiny 0}; \node at ($(a)+(17,1)$) {\tiny 0}; 
\node at ($(a)+(18,1)$) {\tiny 0}; \node at ($(a)+(19,1)$) {\tiny 0}; \node at ($(a)+(20,1)$) {\tiny 0}; \node at ($(a)+(21,1)$) {\tiny 0}; \node at ($(a)+(22,1)$) {\tiny 0}; \node at ($(a)+(23,1)$) {\tiny 0}; \node at ($(a)+(24,1)$) {\tiny 0}; \node at ($(a)+(25,1)$) {\tiny 0}; 

\node[color=blue] at ($(a)+(-1.5,2)$) {\tiny a=6};
\node at ($(a)+(0,2)$) {\tiny 0}; \node at ($(a)+(1,2)$) {\tiny 0}; \node at ($(a)+(2,2)$) {\tiny 0}; \node at ($(a)+(3,2)$) {\tiny 0}; \node at ($(a)+(4,2)$) {\tiny 0}; \node at ($(a)+(5,2)$) {\tiny 0}; 
\node at ($(a)+(6,2)$) {\tiny 0}; \node at ($(a)+(7,2)$) {\tiny 0}; \node at ($(a)+(8,2)$) {\tiny 0}; \node at ($(a)+(9,2)$) {\tiny 0}; \node at ($(a)+(10,2)$) {\tiny 0}; \node at ($(a)+(11,2)$) {\tiny 0}; 
\node at ($(a)+(12,2)$) {\tiny 0}; \node at ($(a)+(13,2)$) {\tiny 0}; \node at ($(a)+(14,2)$) {\tiny 0}; \node at ($(a)+(15,2)$) {\tiny 0}; \node at ($(a)+(16,2)$) {\tiny 0}; \node at ($(a)+(17,2)$) {\tiny 0}; 
\node at ($(a)+(18,2)$) {\tiny 0}; \node at ($(a)+(19,2)$) {\tiny 0}; \node at ($(a)+(20,2)$) {\tiny 0}; \node at ($(a)+(21,2)$) {\tiny 0}; \node at ($(a)+(22,2)$) {\tiny 0}; \node at ($(a)+(23,2)$) {\tiny 0}; \node at ($(a)+(24,2)$) {\tiny 0}; \node at ($(a)+(25,2)$) {\tiny 0}; 

\node[color=blue] at ($(a)+(-1.5,3)$) {\tiny a=5};
\node at ($(a)+(0,3)$) {\tiny 0}; \node at ($(a)+(1,3)$) {\tiny 0}; \node at ($(a)+(2,3)$) {\tiny 0}; \node at ($(a)+(3,3)$) {\tiny 0}; \node at ($(a)+(4,3)$) {\tiny 0}; \node at ($(a)+(5,3)$) {\tiny 0}; 
\node at ($(a)+(6,3)$) {\tiny 0}; \node at ($(a)+(7,3)$) {\tiny 0}; \node at ($(a)+(8,3)$) {\tiny 0}; \node at ($(a)+(9,3)$) {\tiny 0}; \node at ($(a)+(10,3)$) {\tiny 0}; \node at ($(a)+(11,3)$) {\tiny 0}; 
\node at ($(a)+(12,3)$) {\tiny 0}; \node at ($(a)+(13,3)$) {\tiny 0}; \node at ($(a)+(14,3)$) {\tiny 0}; \node at ($(a)+(15,3)$) {\tiny 0}; \node at ($(a)+(16,3)$) {\tiny 0}; \node at ($(a)+(17,3)$) {\tiny 0}; 
\node at ($(a)+(18,3)$) {\tiny 0}; \node at ($(a)+(19,3)$) {\tiny 0}; \node at ($(a)+(20,3)$) {\tiny 0}; \node at ($(a)+(21,3)$) {\tiny 0}; \node at ($(a)+(22,3)$) {\tiny 0}; \node at ($(a)+(23,3)$) {\tiny 0}; \node at ($(a)+(24,3)$) {\tiny 0}; \node at ($(a)+(25,3)$) {\tiny 0}; 

\node[color=blue] at ($(a)+(-1.5,4)$) {\tiny a=4};
\node at ($(a)+(0,4)$) {\tiny 3}; \node at ($(a)+(1,4)$) {\tiny 3}; \node at ($(a)+(2,4)$) {\tiny 3}; \node at ($(a)+(3,4)$) {\tiny 3}; \node at ($(a)+(4,4)$) {\tiny 3}; \node at ($(a)+(5,4)$) {\tiny 3}; 
\node at ($(a)+(6,4)$) {\tiny 3}; \node at ($(a)+(7,4)$) {\tiny 3}; \node at ($(a)+(8,4)$) {\tiny 3}; \node at ($(a)+(9,4)$) {\tiny 3}; \node at ($(a)+(10,4)$) {\tiny 3}; \node at ($(a)+(11,4)$) {\tiny 3}; 
\node at ($(a)+(12,4)$) {\tiny 3}; \node at ($(a)+(13,4)$) {\tiny 3}; \node at ($(a)+(14,4)$) {\tiny 3}; \node at ($(a)+(15,4)$) {\tiny 3}; \node at ($(a)+(16,4)$) {\tiny 3}; \node at ($(a)+(17,4)$) {\tiny 3}; 
\node at ($(a)+(18,4)$) {\tiny 3}; \node at ($(a)+(19,4)$) {\tiny 3}; \node at ($(a)+(20,4)$) {\tiny 3}; \node at ($(a)+(21,4)$) {\tiny 3}; \node at ($(a)+(22,4)$) {\tiny 3}; \node at ($(a)+(23,4)$) {\tiny 3}; \node at ($(a)+(24,4)$) {\tiny 3}; \node at ($(a)+(25,4)$) {\tiny 3}; 

\node[color=blue] at ($(a)+(-1.5,5)$) {\tiny a=3};
\node at ($(a)+(0,5)$) {\tiny 0}; \node at ($(a)+(1,5)$) {\tiny 8}; \node at ($(a)+(2,5)$) {\tiny 4}; \node at ($(a)+(3,5)$) {\tiny 7}; \node at ($(a)+(4,5)$) {\tiny 3}; \node at ($(a)+(5,5)$) {\tiny 0}; 
\node at ($(a)+(6,5)$) {\tiny 6}; \node at ($(a)+(7,5)$) {\tiny 7}; \node at ($(a)+(8,5)$) {\tiny 0}; \node at ($(a)+(9,5)$) {\tiny 4}; \node at ($(a)+(10,5)$) {\tiny 5}; \node at ($(a)+(11,5)$) {\tiny 0}; 
\node at ($(a)+(12,5)$) {\tiny 8}; \node at ($(a)+(13,5)$) {\tiny 4}; \node at ($(a)+(14,5)$) {\tiny 7}; \node at ($(a)+(15,5)$) {\tiny 3}; \node at ($(a)+(16,5)$) {\tiny 0}; \node at ($(a)+(17,5)$) {\tiny 6}; 
\node at ($(a)+(18,5)$) {\tiny 7}; \node at ($(a)+(19,5)$) {\tiny 0}; \node at ($(a)+(20,5)$) {\tiny 4}; \node at ($(a)+(21,5)$) {\tiny 5}; \node at ($(a)+(22,5)$) {\tiny 0}; \node at ($(a)+(23,5)$) {\tiny 8}; \node at ($(a)+(24,5)$) {\tiny 4}; \node at ($(a)+(25,5)$) {\tiny 7}; 

\node[color=blue] at ($(a)+(-1.5,6)$) {\tiny a=2};
\node at ($(a)+(0,6)$) {\tiny 0}; \node at ($(a)+(1,6)$) {\tiny 0}; \node at ($(a)+(2,6)$) {\tiny 3}; \node at ($(a)+(3,6)$) {\tiny 4}; \node at ($(a)+(4,6)$) {\tiny 3}; \node at ($(a)+(5,6)$) {\tiny 0}; 
\node at ($(a)+(6,6)$) {\tiny 0}; \node at ($(a)+(7,6)$) {\tiny 6}; \node at ($(a)+(8,6)$) {\tiny 0}; \node at ($(a)+(9,6)$) {\tiny 0}; \node at ($(a)+(10,6)$) {\tiny 6}; \node at ($(a)+(11,6)$) {\tiny 0}; 
\node at ($(a)+(12,6)$) {\tiny 0}; \node at ($(a)+(13,6)$) {\tiny 3}; \node at ($(a)+(14,6)$) {\tiny 4}; \node at ($(a)+(15,6)$) {\tiny 3}; \node at ($(a)+(16,6)$) {\tiny 0}; \node at ($(a)+(17,6)$) {\tiny 0}; 
\node at ($(a)+(18,6)$) {\tiny 6}; \node at ($(a)+(19,6)$) {\tiny 0}; \node at ($(a)+(20,6)$) {\tiny 0}; \node at ($(a)+(21,6)$) {\tiny 6}; \node at ($(a)+(22,6)$) {\tiny 0}; \node at ($(a)+(23,6)$) {\tiny 0}; \node at ($(a)+(24,6)$) {\tiny 3}; \node at ($(a)+(25,6)$) {\tiny 4}; 

\node[color=blue] at ($(a)+(-1.5,7)$) {\tiny a=1};
\node at ($(a)+(0,7)$) {\tiny 0}; \node at ($(a)+(1,7)$) {\tiny 0}; \node at ($(a)+(2,7)$) {\tiny 0}; \node at ($(a)+(3,7)$) {\tiny 8}; \node at ($(a)+(4,7)$) {\tiny 3}; \node at ($(a)+(5,7)$) {\tiny 0}; 
\node at ($(a)+(6,7)$) {\tiny 0}; \node at ($(a)+(7,7)$) {\tiny 0}; \node at ($(a)+(8,7)$) {\tiny 0}; \node at ($(a)+(9,7)$) {\tiny 0}; \node at ($(a)+(10,7)$) {\tiny 0}; \node at ($(a)+(11,7)$) {\tiny 0}; 
\node at ($(a)+(12,7)$) {\tiny 0}; \node at ($(a)+(13,7)$) {\tiny 0}; \node at ($(a)+(14,7)$) {\tiny 8}; \node at ($(a)+(15,7)$) {\tiny 3}; \node at ($(a)+(16,7)$) {\tiny 0}; \node at ($(a)+(17,7)$) {\tiny 0}; 
\node at ($(a)+(18,7)$) {\tiny 0}; \node at ($(a)+(19,7)$) {\tiny 0}; \node at ($(a)+(20,7)$) {\tiny 0}; \node at ($(a)+(21,7)$) {\tiny 0}; \node at ($(a)+(22,7)$) {\tiny 0}; \node at ($(a)+(23,7)$) {\tiny 0}; \node at ($(a)+(24,7)$) {\tiny 0}; \node at ($(a)+(25,7)$) {\tiny 8}; 

\node[color=blue] at ($(a)+(-1.5,8)$) {\tiny a=0};
\node at ($(a)+(0,8)$) {\tiny 0}; \node at ($(a)+(1,8)$) {\tiny 0}; \node at ($(a)+(2,8)$) {\tiny 0}; \node at ($(a)+(3,8)$) {\tiny 0}; \node at ($(a)+(4,8)$) {\tiny 3}; \node at ($(a)+(5,8)$) {\tiny 0}; 
\node at ($(a)+(6,8)$) {\tiny 0}; \node at ($(a)+(7,8)$) {\tiny 0}; \node at ($(a)+(8,8)$) {\tiny 0}; \node at ($(a)+(9,8)$) {\tiny 0}; \node at ($(a)+(10,8)$) {\tiny 0}; \node at ($(a)+(11,8)$) {\tiny 0}; 
\node at ($(a)+(12,8)$) {\tiny 0}; \node at ($(a)+(13,8)$) {\tiny 0}; \node at ($(a)+(14,8)$) {\tiny 0}; \node at ($(a)+(15,8)$) {\tiny 3}; \node at ($(a)+(16,8)$) {\tiny 0}; \node at ($(a)+(17,8)$) {\tiny 0}; 
\node at ($(a)+(18,8)$) {\tiny 0}; \node at ($(a)+(19,8)$) {\tiny 0}; \node at ($(a)+(20,8)$) {\tiny 0}; \node at ($(a)+(21,8)$) {\tiny 0}; \node at ($(a)+(22,8)$) {\tiny 0}; \node at ($(a)+(23,8)$) {\tiny 0}; \node at ($(a)+(24,8)$) {\tiny 0};  \node at ($(a)+(25,8)$) {\tiny 0}; 

\end{tikzpicture}
\caption{Tables of $S_1(a,b,-)$, $S_2(a,b,3)$, $S_3(a,b,3)$ values for $p=11$ and small integers $a,b$. Yellow shading indicates the range covered by Theorem~\ref{thm nd}, and the dotted lines enclose the  region covered by Theorem~\ref{thm SM}.
The structure of the gray shading is discussed in Section \ref{sec 5.3}.
}
\end{figure}
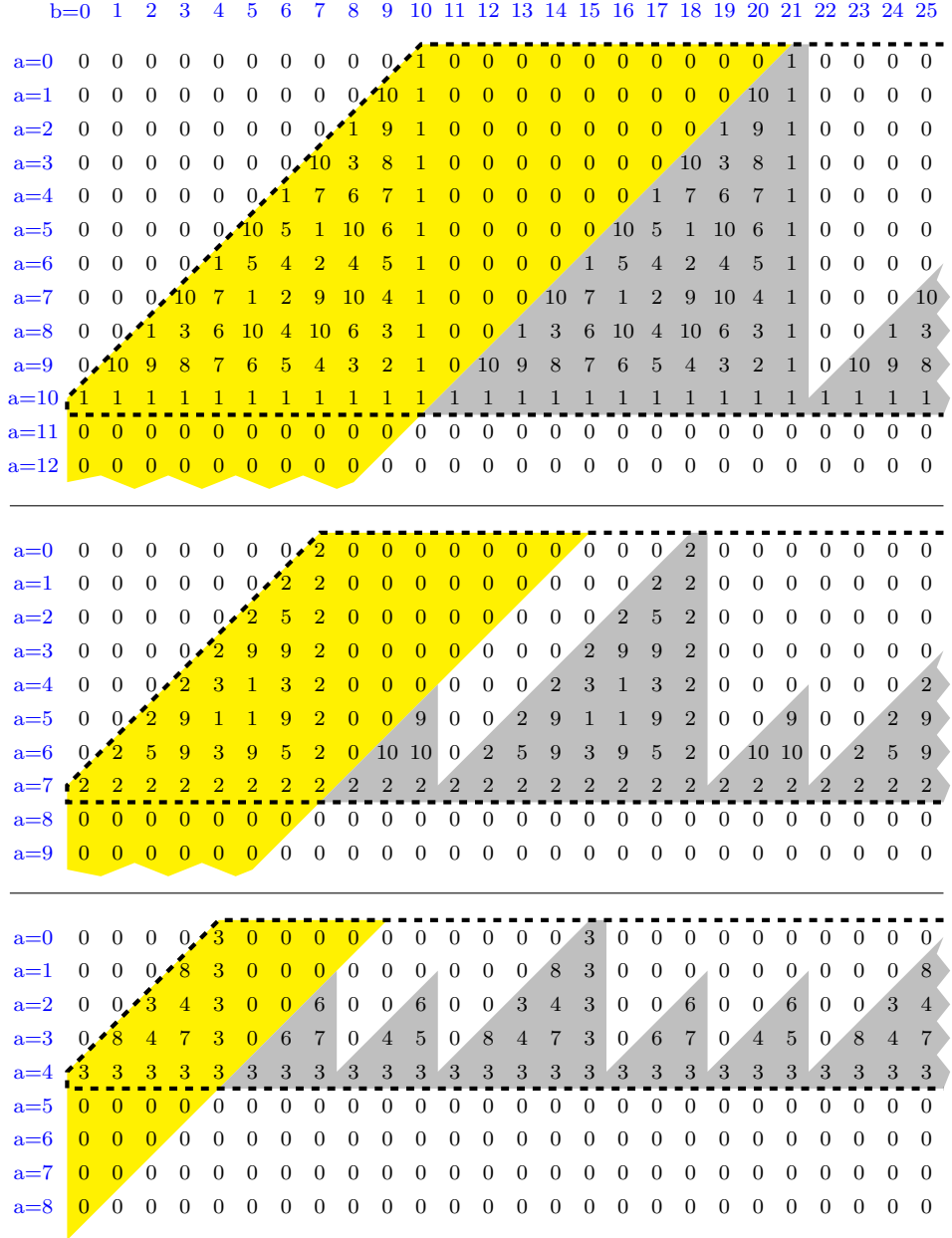

\subsection{Factorization properties}
\label{sec 5.3}

By Lemmas \ref{lem b=b+p} and \ref{lem S=0} we have
$S_n(a,b+p,c) = S_n(a,b,c)$ and $S_n(a,b,c)=0$ if
$a\geq p-(n-1)c$. Thus, for given $c$, it is enough to analyze $S_n(a,b,c)$ in the rectangle
$\Om=\{ (a,b) \ | \ a\in [0, p-1-(n-1)c], b\in [0,p-1]\}$. This rectangle is partitioned into $n$ smaller
rectangles :
\bea
\Om_0(n,c)
&=&
\{ (a,b) \ | \ a\in [0, \ p-1-(n-1)c], \ \ b\in [0,\ p-1-(n-1)c]\}\,,
\\
\Om_i(n,c)
&=&
\{ (a,b) \ | \ a\in [0,\ p-1-(n-1)c], 
\\
&&
\phantom{aaaaaaa}
b\in [p-1-(n-i)c+1, \ p-1-(n-i-1)c]\}\,,
\quad i=1,\dots, n-1,
\eea
see the
tables in Figure 1. The values of $S_n(a,b,c)$ in $\Om_0(n,c)$ are given by Theorem \ref{thm nd} and Lemma \ref{lem S=01}.
The values of $S_n(a,b,c)$ in a rectangle $\Om_i(n,c)$ are given by Theorem \ref{thm nd} and Lemma \ref{lem S=01} also,
but
applied to $\F_p$-Selberg integrals of smaller dimensions with the same value of
$c$ and suitable choices of values for $a$ and $b$.
Namely, we have the following factorization property.

\begin{thm}
\label{thm facto} For   $(a,b)\in \Om_i(n,c)$ with $i>0$, we have
\bean
\label{factor}
S_n(a,b,c) \,&=&\, (-1)^{(n-i)ic}\,\binom{nc}{ic}\, 
\\
\notag
&\times&
\frac{\prod_{j=1}^{n-i}\binom{p-1-(n-j)c-a}{(j-1)c}\,\prod_{j=1}^{i}\binom{p-1-(n-j)c-a}{(j-1)c}}
{\prod_{j=1}^{n}\binom{p-1-(n-j)c-a}{(j-1)c}}\,
\\
\notag
&\times &
S_{n-i}(a+ic,b,c) \, S_{i}(a+(n-i)c, b+(n-i)c -p,c)\,.
\eean

\end{thm}  

Notice that all binomials $\binom{\al}{\beta}$
in the second line of \eqref{factor} have $p>\al \geq \beta\geq 0$.
Notice also that $(a+ic,b) \in \Om_0(n-i,c)$ and $(a+(n-i)c,b+(n-i)c-p)\in \Om_0(i,c)$,
 and hence 
Theorem \ref{thm nd} and Lemma \ref{lem S=01} can be applied to 
$S_{n-i}(a+ic,b,c)$ and $S_{i}(a+(n-i)c, b+(n-i)c -p,c)$\,.

\begin{proof}
The theorem follows from formula \eqref{SMnb} and Lucas' Theorem.
\end{proof}

\section{A remarkable combinatorial identity}
\label{sec 6}

In this section we sketch another proof of Theorem \ref{thm nd}. We do this because at the heart of this proof there is a remarkable identity (Theorem \ref{thm:identity}) for polynomials in two variables. 

\smallskip

\noindent{\bf Notation.}  Let $c, n$ be positive integers. For $1\leq i< j\leq n$ we will consider non-negative integers $0 \leq m_{ij} \leq 2c$ and we set $\bm_{ij}=2c-m_{ij}$. For $1\leq k\leq n$ define 
\[ r_k=\sum_{1\leq i<k} \bm_{ik} + \sum_{k<i\leq n} m_{ki}, \qquad
   s_k=\sum_{1\leq i<k} m_{ik} + \sum_{k<i\leq n} \bm_{ki}.
 \]
We will use the (rising) Pochhammer symbol $(x)_m=x(x+1)(x+2)\cdots(x+m-1)$.

\begin{thm} \label{thm:identity}
Let $n\geq 2,\ c\geq 1$ be positive integers. In $\mathbb Z[x,y]$ we have the identity 
\begin{multline*}
\sum_{\mm} \left(
  (-1)^{\sum_{i<j} m_{ij}} \prod_{i<j}
\binom{2c}{m_{ij}} \cdot
\prod_{k=1}^n 
(x)_{r_k}
(y)_{s_k}
 \right)\\
 =\
\prod_{k=1}^{n-1}\frac{((k+1)c)!}{c!}(x)_{kc}(y)_{kc}
(x+y+(2n-k-2)c)_{kc}\ ,
\end{multline*}
where by $\sum_{\mm}$ we mean the $\binom{n}{2}$-fold summation
$
\sum_{m_{12}=0}^{2c} 
\sum_{m_{13}=0}^{2c} 
\sum_{m_{14}=0}^{2c} 
\ldots
\sum_{m_{n-1,n}=0}^{2c} 
$.
\end{thm}
The summands of the left-hand side are of degree $4c\binom{n}{2}$ polynomials, and according to the theorem, their sum is 
the right-hand side, which is the product of degree $3c\binom{n}{2}$, with linear factors. The reader is invited to verify that for $n=2$ the theorem reduces to a hypergeometric identity, namely Dixon's Theorem (\cite[Theorem 3.4.1]{AAR}) on the factorization of ${}_3F_2$ with certain parameters. For instance the $n=2, c=2$ case of
Theorem \ref{thm:identity} states that the sum of the terms
\begin{multline*}
(x+2)(x+3)(y+2)(y+3), \qquad -4x(x+2)y(y+2), \qquad 6x(x+1)y(y+1),\\
-4x(x+2)y(y+2), \qquad (x+2)(x+3)(y+2)(y+3)
\end{multline*}
is $12(x+y+2)(x+y+3)$ (here we canceled the factor $xy(x+1)(y+1)$, which
 appears in each term and on the right-hand side as well).
The explicit form of the identity for $n=3$ is
\begin{multline*}
\sum_{m_{12}, m_{23}, m_{13}=0}^{2c} (-1)^{m_{12}+m_{13}+m_{23}} 
\binom{2c}{m_{12}}\binom{2c}{m_{23}} \binom{2c}{m_{13}} 
\\
\times
\prod_{k=0}^{m_{12}+m_{13}-1} (x+k)
\prod_{k=0}^{2c-m_{12}+m_{23}-1}(x+k)
\prod_{k=0}^{4c-m_{13}-m_{23}-1}(x+k) 
\\
\times
\prod_{k=0}^{4c-m_{12}-m_{13}-1} (y+k)
\prod_{k=0}^{2c-m_{23}+m_{12}-1} (y+k)
\prod_{k=0}^{m_{13}+m_{23}-1}(y+j)
\\
\ \hskip 5.2 true cm =\ \frac{(2c)!}{c!}\,\ \frac{(3c)!}{c!}\ 
\prod_{k=1}^c (x+k-1) (y+k-1) (x+y+4c-k)
\\
\times \ \prod_{k=1}^{2c} (x+k-1) (y+k-1)(x+y+4c-k).
\end{multline*}

\medskip

\noindent {\em Sketch of the proof of Theorem \ref{thm:identity}.}  Consider equation \eqref{cSn} for a positive integer $c$, that is, the classical Selberg integral formula in $n$ dimensions. On the left-hand side we {\em decouple} the variables, i.e. we substitute 
$
(x_i-x_j)^{2c}=\sum_{m_{ij}=0}^{2c} \binom{2c}{m_{ij}} x_i^{m_{ij}}(-x_j)^{\bm_{ij}}.
$
We obtain
\begin{multline*}
\sum_{\mm} \left(
  (-1)^{\sum_{i<j} m_{ij}} \prod_{i<j}
\binom{2c}{m_{ij}} \cdot
\prod_{k=1}^n \left( \int_0^1 x_k^{a+r_k}(1-x_k)^b dx_k\right)\right) \\
=\ \prod_{j=1}^n \frac{(jc)!}{c!} \frac{(a+(j-1)c)!(b+(j-1)c)!}{(a+b+(n+j-2)c+1)!}.
\end{multline*}
Now writing $\Gamma(a+r_k+1)\Gamma(b+1)/\Gamma(a+r_k+b+2)$ for the one-dimensional Selberg integrals
 on the left-hand side, and substituting 
\[
x=a+1 ,\qquad y=-(a+2(n-1)c+b+1)\,,
\]
the obtained identity rearranges to the statement in the theorem. \qed

\medskip

We believe that the identity in Theorem \ref{thm:identity} is interesting on it own right, but here is a sketch how to use it to prove Theorem \ref{thm nd}.

\smallskip

Consider the left-hand side of \eqref{main n}, and carry out the same {\em decoupling} 
of variables as we did in the proof of Theorem \ref{thm:identity}. 
We obtain a sum, parameterized by choices of $m_{ij}$, and in each summand we get a product of one-dimensional $\F_p$-Selberg integrals of the form
$
\int_{[1]_p} x_k^{A_k}(1-x_k)^b dx_k
$
for some $A_k$. Substituting the  value  $-A_k!b!/(A_k+b+1-p)!$ for such 
a one-dimensional integral (formula \eqref{bf1}),
we obtain an explicit formula (no integrals anymore!) for the left-hand side of \eqref{main n}. 
The summation Theorem~\ref{thm:identity}  brings that sum to a product form, and one obtains exactly the 
right-hand side of \eqref{main n}. 

In this proof one has to pay additional attention to the case $a+b < p-1$, when
some integrals $\int_{[1]_p} x_k^{A_k}(1-x_k)^b dx_k$ have $A_k+b<p-1$ and  are equal to zero by formula
\eqref{bf2}.
Still in this case the sum of nonzero terms is transformed to the desired product by the identity of Theorem 
\ref{thm:identity} with parameter $c$ replaced by $d:= a+b+(n-1)c+1-p$.
\qed

\section{KZ equations}
\label{sec 7}

\subsection{Special case of $\slt$ KZ equations over $\C$}
\label{sec ckz}

Let $e,f,h$ be the standard basis of the complex Lie algeba $\slt$ with $[e,f]=h$, $[h,e]=2e$, $[h,f]=-2f$.
The element  
\bean
\label{Casimir}
\Omega =  e \otimes f + f \otimes e +
                       \frac{1}{2} h \otimes h\  \in\  \slt \ox\slt
                       \eean
 is called the Casimir element.  
For  $i\in\Z_{\geq 0}$ let  $V_i$ be the irreducible $i+1$-dimensional $\slt$-module with
basis $v_i, fv_i,\dots,f^iv_i$ such that $ev_i=0$, $hv_i=iv_i$.

Let $u(z_1,z_2)$ be a function taking values in $V_{m_1}\ox V_{m_2}$ and solving the KZ equations
\bean
\label{KZ}
\ka\,\frac{\der u}{\der z_1} = \frac{\Om}{z_1-z_2}\,u\,,
\qquad
\ka\,\frac{\der u}{\der z_2} = \frac{\Om}{z_2-z_1}\,u\,,
\eean
where $\ka \in\C^\times$ is a parameter of the equations. Let
$\Sing[m_1+m_2-2n]$ denote the space of singular vectors of
weight $m_1+m_2-2n$  in $V_{m_1}\ox V_{m_2}$,
\bea
\Sing[m_1+m_2-2n] = \{v\in V_{m_1}\ox V_{m_2}\ | \ hv = (m_1+m_2-2n)v, \ ev=0\}.
\eea
This space is one-dimensional if the integer $n$ satisfies 
$0\leq n\leq \on{min}(m_1,m_2)$ and is zero-dimensional otherwise.
According to \cite{SV1}, solutions $u$ with values in 
$\Sing[m_1+m_2-2n]$ are expressible in terms of $n$-dimensional hypergeometric integrals
\bea
u(z_1,z_2) = \sum_r u_r(z_1,z_2)\,f^{r}v_{m_1}\ox f^{n-r}v_{m_2}
\eea
with 
\bea
u_r(z_1,z_2)= (z_1-z_2)^{m_1m_2/2\ka}\int_C W_r(z_1,z_2,t) \Psi(z_1,z_2,t)  \,dt_1\dots dt_n\,.
\eea
Here the domain of integration is  the simplex $C=\{ t\in\R^n\ |\ z_1\leq t_n\leq \dots\leq t_1\leq z_2\}$.
The function $\Psi(z_1,z_2, t)$ is called the {\it master function},
\bea
\Psi(z_1,z_2, t) = \prod_{1\leq i<j\leq n} (t_i-t_j)^{2/\ka} \prod_{i=1}^n(t_i-z_1)^{-m_1/\ka}(t_i-z_2)^{-m_2/\ka}\,,
\eea
the rational functions $W_r(z_1,z_2,t)$ are called the {\it weight functions},
\bea
W_r(z_1,z_2,t) = \sum_{J\subset \{1,\dots,n\}
\atop |J|=r} \
\prod_{j\in J}\ \frac1{t_j-z_1}\ \prod_{j\notin J}\frac 1{t_i-z_2}\,.
\eea 
The fact that $u$ is a solution in $\Sing[m_1+m_2-2n]$ implies that
\bean
\label{ev=0}
\phantom{aaa}
(n-r)(m_2-n+r+1) u_r\, + \,(r+1)(m_1-r) u_{r+1} \,= \,0,\qquad r=1,\dots,n-1.
\eean
The coordinate functions $u_r$ are generalizations of the Selberg integral. In fact, $u_0$ and $u_n$
are exactly the Selberg integrals. For example,
\bea
u_0(z_1,z_2) = (z_1-z_2)^{m_1m_2/2\ka}\!\! \int_C
\prod_{1\leq i<j\leq n} \!\!\! (t_i-t_j)^{2/\ka} \prod_{i=1}^n(t_i-z_1)^{-m_1/\ka}(t_i-z_2)^{-m_2/\ka-1}dt_1\dots dt_n\,.
\eea
The change of variables $t_i = (z_2-z_1)s_i+z_1$ for $i=1,\dots, n$ gives

\bea
u_0(z_1,z_2) = \frac{(-1)^A(z_1-z_2)^B}{n!}\,\tilde S_n\Big(1-\frac{m_1}\ka, -\frac{m_2}\ka,\frac{1}\ka\Big),
\eea
\vsk.3>

\noindent
where $\tilde S_n(\al,\beta,\ga)$ denotes the Selberg integral  \eqref{clS},
  $A=\frac{n(n-1-m_1)}\ka+n$, 
  \\
$B=\frac{m_1m_2-2n(m_1+m_2) + 2n(n-1)}{2\ka}$.
By formula \eqref{ev=0}, we obtain

\bean
\label{KZ n=2}
u(z_1,z_2) &=& \ka^n\,\frac{(-1)^A(z_1-z_2)^B}{n!}\, \prod_{j=1}^n \frac{\Ga(1+\frac{j}\ka)}{\Ga(1+\frac{1}\ka)}\,
\frac{\Ga(1-\frac{m_1-j+1}\ka)\,\Ga(1-\frac{m_2-j+1}\ka)}
{\Ga(1-\frac{m_1+m_2-n-j+2}\ka)}\,,
\\
\notag
&\times &
\sum_{r=0}^n (-1)^r\binom{n}{r}
\frac{f^rv_1\ox f^{n-r}v_2}{\prod_{j=1}^r(m_1-j+1) \prod_{j=1}^{n-r}(m_2-j+1)}\,.
\eean

\subsection{Special case of $\slt$ KZ equations over $\F_p$}

 Let $p$ be an odd prime number. Let $\ka$ be a ratio of two integers not divisible by $p$.
Let $m_1, m_2 $ be positive integers such that $m_1,m_2<p$. Consider the Lie algebra $\slt$ over the field $\F_p$.
Let $V^p_{m_1}$, $V^p_{m_2}$  be the $\slt$-modules over $\F_p$,
corresponding to the complex representations
$V_{m_1}$, $V_{m_2}$.
Then the KZ differential equations \eqref{KZ} with values in
$V^p_{m_1}\ox V^p_{m_2}$ are well-defined, and 
we may discuss their polynomial solutions in variables $z_1,z_2$.
Let
\bea
\Sing[m_1+m_2-2n]_p = \{v\in V_{m_1}^p\ox V_{m_2}^p\ | \ hv = (m_1+m_2-2n)v, \ ev=0\}.
\eea
This space is one-dimensional, if the integer $n$ satisfies 
$0\leq n\leq \on{min}(m_1,m_2)$ and is zero-dimensional otherwise.

\vsk.2>

Choose the least positive integers $M_1, M_2, M_{12}, c$ such that
\bean
\label{MM}
M_i\equiv -\frac{m_i}\ka\,, \quad
M_{12}\equiv \frac{m_1m_2}{2\ka},\quad c\equiv \frac1\ka\, \qquad (\on{mod}\,p).
\eean
According to \cite{SV2}, solutions $u$ with values in 
$\Sing[m_1+m_2-2n]_p$ are expressible in terms of $n$-dimensional $\F_p$-hypergeometric integrals
\bean
\label{up}
u(z_1,z_2) = \sum_r u_r(z_1,z_2)\,f^{r}v_{m_1}\ox f^{n-r}v_{m_2}
\eean
with 
\bea
u_r(z_1,z_2)= (z_1-z_2)^{M_{12}}\int_{[1,\dots,1]_p} W_r(z_1,z_2,t) \Psi_p(z_1,z_2,t)  \,dt_1\dots dt_n\,,
\eea
where $\Psi_p(z_1,z_2, t)$ is the {\it master polynomial},
\bea
\Psi_p(z_1,z_2, t) = \prod_{1\leq i<j\leq n} (t_i-t_j)^{2c} \prod_{i=1}^n(t_i-z_1)^{M_1}(t_i-z_2)^{M_2}\,.
\eea

\begin{thm}
\label{thm KZ n=2}
Assume that  $M_1,M_2, M_{12}, c, n$ are positive integers such that
\bean
\label{abc n}
&&
M_1+(n-1)c < p,
\qquad \phantom{aaaaa}
M_2+(n-1)c < p,
 \\
 \notag
 &&
 p\leq M_1+M_2+(n-1)c,
\qquad M_1+M_2+(2n-2)c < 2p-1\ .
\eean
Then  the function $u(z_1,z_2)$, defined by \eqref{up}, is given by the formula
\bean
\label{KZ n=2 p}
u(z_1,z_2) &=& (-1)^A (z_1-z_2)^B\, \prod_{j=1}^n
 \frac{(jc)!}{c!}\,
\frac{(M_1+(j-1)c)!\, (M_2+(j-1)c)!}
{(M_1+M_2 +(n+j-2)c-p)!}\,,
\\
\notag
&\times &
\sum_{r=0}^n (-1)^r\binom{n}{r}
\frac{f^rv_1\ox f^{n-r}v_2}{\prod_{j=1}^r(M_1+(j-1)c) \prod_{j=1}^{n-r}(M_2+(j-1)c)}\,,
\eean
where
\bea
A=n(M_1+(n-1)c +1)\,,\qquad B= M_{12}+n(M_1+M_2+(n-1)c-p)\,.
\eea
\end{thm}

For $n=1$ this is  \cite[Theorem 4.3]{V7}.

\begin{proof}
The proof follows from the $\F_p$-Selberg integral formula of Theorem \ref{thm nd}
and formula \eqref{ev=0}, cf.
Section \ref{sec ckz}.
\end{proof}

\end{document}